\documentclass[reqno,12pt,a4paper]{amsart}

\usepackage{tikz}
\usetikzlibrary{matrix,arrows,calc,snakes,patterns,decorations.markings}

\usepackage[mathscr]{eucal}
\usepackage{graphics,epic}
\usepackage{amsfonts}
\usepackage{amscd}
\usepackage{latexsym}
\usepackage{amsmath,amssymb, amsthm, stmaryrd, bm, bbm}
\usepackage[all,2cell]{xy}

\newtheorem{theorem}{Theorem}[section]
\newtheorem*{theorem*}{Theorem}

\newtheorem{lemma}[theorem]{Lemma}
\newtheorem{proposition}[theorem]{Proposition}
\newtheorem{corollary}[theorem]{Corollary}

\newtheorem*{conjecture*}{Conjecture}
\newtheorem{question}[theorem]{Question}
\newtheorem*{question*}{Question}
\theoremstyle{remark}
\newtheorem{remark}[theorem]{Remark}
\newtheorem{example}[theorem]{Example}
\theoremstyle{definition}
\newtheorem{definition}[theorem]{Definition}
\newtheorem{construction}[theorem]{Construction}

\newcommand{\ie}{{\em i.e.~}\ }
\newcommand{\confer}{{\em cf.~}\ }

\newcommand{\opname}[1]{\operatorname{\mathsf{#1}}}

\renewcommand{\mod}{\opname{mod}\nolimits}
\newcommand{\fdmod}{\opname{fdmod}\nolimits}

\newcommand{\proj}{\opname{proj}\nolimits}

\newcommand{\Mod}{\opname{Mod}\nolimits}
\newcommand{\pd}{\opname{pd}\nolimits}

\newcommand{\add}{\opname{add}\nolimits}

\renewcommand{\ker}{\opname{ker}\nolimits}

\newcommand{\thick}{\opname{thick}\nolimits}

\newcommand{\per}{\opname{per}\nolimits}

\newcommand{\std}{\mathrm{std}}

\newcommand{\Hom}{\opname{Hom}}
\newcommand{\End}{\opname{End}}

\newcommand{\RHom}{\mathbb{R}\hspace{-2pt}\opname{Hom}}

\newcommand{\cHom}{{\mathcal{H}\it{om}}}
\newcommand{\cEnd}{{\mathcal{E}\it{nd}}}

\newcommand{\Ext}{\opname{Ext}}

\newcommand{\ten}{\otimes}
\newcommand{\lten}{{\ten}^{\mathbb{L}}}

\newcommand{\Tor}{\opname{Tor}}

\newcommand{\Cone}{\opname{Cone}}

\newcommand{\ca}{{\mathcal A}}
\newcommand{\cb}{{\mathcal B}}

\newcommand{\cd}{{\mathcal D}}

\newcommand{\ck}{{\mathcal K}}

\newcommand{\cs}{{\mathcal S}}
\newcommand{\ct}{{\mathcal T}}

\newcommand{\cx}{{\mathcal X}}
\newcommand{\cy}{{\mathcal Y}}

\numberwithin{equation}{section}

\begin{document}

\title[The silting theorem]{A differential graded approach to the silting theorem}

\author{Zongzhen Xie, Dong Yang and Houjun Zhang}
\dedicatory{Dedicated to Steffen Koenig on the occasion of his 60th birthday}
\address{Zongzhen Xie, Department of Mathematics and Computer Science, School of Biomedical Engineering and Informatics, Nanjing Medical University, Nanjing, 211166, P. R. China.}

\email{zzhx@njmu.edu.cn}
\address{Dong Yang, Department of Mathematics, Nanjing University, Nanjing 210093, P. R. China}

\email{yangdong@nju.edu.cn}
\address{Houjun Zhang, Department of Mathematics, Nanjing University, Nanjing 210093, P. R. China}

\email{zhanghoujun@nju.edu.cn}

\begin{abstract} A silting theorem was established by Buan and Zhou as a generalisation of the classical  tilting theorem of Brenner and Butler. In this paper, we give an alternative proof of the theorem by using differential graded algebras.
\end{abstract}
\maketitle

\section{Introduction}\label{s:introduction}
\medskip
Let $k$ be a field and $A$ be a finite-dimensional $k$-algebra. Let $T$ be a classical tilting module over $A$ and put $B=\End_{A}(T)$. Then there are two adjoint pairs of functors between the two categories of finite-dimensional modules
\[
\xymatrix{
\mod A\ar@<.7ex>[d]^{\Hom_A(T,?)}\\
\mod B\ar@<.7ex>[u]^{?\ten_B T}
}\qquad\qquad
\xymatrix{
\mod A\ar@<.7ex>[d]^{\Ext^1_A(T,?)}\\
\mod B\ar@<.7ex>[u]^{\Tor_1^B(?,T)}
}
\]
Put $\mathscr{T}_T=\ker(\Ext^1_A(T,?)),~\mathscr{F}_T=\ker(\Hom_A(T,?)),~\mathscr{X}_T=\ker(?\ten_B T)$ and $\mathscr{Y}_T=\ker(\Tor^B_1(?,T))$.
The classical tilting theorem of Brenner and Butler \cite{BrennerButler80,HappelRingel82} states that  $(\mathscr{T}_T,\mathscr{F}_T)$ is a torsion pair of $\mod A$, $(\mathscr{X}_T,\mathscr{Y}_T)$ is a torsion pair of $\mod B$, and the restrictions of the above four functors are equivalences
\[
\xymatrix{
\mathscr{T}_T\ar@<.7ex>[d]^{\Hom_A(T,?)}\\
\mathscr{Y}_T\ar@<.7ex>[u]^{?\ten_B T}
}\qquad\qquad
\xymatrix{
\mathscr{F}_T\ar@<.7ex>[d]^{\Ext^1_A(T,?)}\\
\mathscr{X}_T\ar@<.7ex>[u]^{\Tor_1^B(?,T)}
}
\]

Recently, Buan and Zhou generalised this theorem to a silting theorem in \cite{BuanZhou16}. Precisely, let $T$ be a 2-term silting complex in the bounded homotopy category $K^{b}(\proj A)$ of finitely generated projective $A$-modules. For example, any projective resolution of a classical tilting module is isomorphic to a 2-term silting complex in $K^b(\proj A)$. Let $\cd^b(\mod A)$ be the bounded derived category of $\mod A$.

\begin{theorem}[{\cite[Theorem 1.1 and Corollary 2.5]{BuanZhou16}}]
\label{thm:buan-zhou-theorem}
\begin{itemize}
\item[(a)] There is a torsion pair $(\mathscr{T}_T,\mathscr{F}_T)$ of $\mod A$ and a torsion pair $(\mathscr{X}_T,\mathscr{Y}_T)$ of $\mod B$ together with equivalences $\Hom_{\cd^b(\mod A)}(T,?)\colon\mathscr{T}_T\stackrel{\simeq}{\longrightarrow} \mathscr{Y}_T$ and $\Hom_{\cd^b(\mod A)}(T,\Sigma?)\colon\mathscr{F}_T\stackrel{\simeq}{\longrightarrow}\mathscr{X}_T$.
\item[(b)] There is a triangle in $K^b(\proj A)$
$$
\begin{CD}
A @>>> T^{\prime}@>f>> T^{\prime\prime}@>>> \Sigma A
\end{CD}
$$
with $T^{\prime}, T^{\prime\prime}\in \add T$. Let $S$ be the mapping cone of 
$$\begin{CD}
\Hom_{K^{b}(\proj A)}(T,T^{\prime})@>\Hom(T,f)>>\Hom_{K^{b}(\proj A)}(T,T^{\prime\prime}).
\end{CD}$$
\item[(c)] $S$ is a 2-term silting complex in $K^{b}(\proj B)$, where $B=\End_{K^b(\proj A)}(T)$.
\item[(d)] There is a surjective algebra homomorphism  $\pi\colon A\rightarrow \bar{A}:=\End_{K^b(\proj B)}(S)$.
\item[(e)] $\pi$ is an isomorphism if and only if $T$ is tilting.
\end{itemize}
Let $\pi^*\colon \mod \bar{A}\hookrightarrow \mod A$ be the restriction functor.
\begin{itemize}
\item[(f)] The functors $\Hom_{\mathcal{D}^{b}(\mod A)}(T,?)$ and $\pi^*\Hom_{\mathcal{D}^{b}(\mod B)}(S,\Sigma ?)$ restrict to inverse equivalences between $\mathscr{T}_T$ and $\mathscr{F}_S$.
\item[(g)] The functors $\Hom_{\mathcal{D}^{b}(\mod A)}(T,\Sigma?)$ and $\pi^*\Hom_{\mathcal{D}^{b}(\mod B)}(S,?)$ restrict to inverse equivalences between $\mathscr{F}_T$ and $\mathscr{T}_S$.
\end{itemize}
\end{theorem}

Part (a) of Theorem~\ref{thm:buan-zhou-theorem} has been generalised to the case when $A$ is a ring and $T$ is a 2-term large silting complex with $\mod$ replaced by $\Mod$ (the category of all modules), see  \cite[Section 4.3]{BreazModoi20}. In this paper, we give an alternative proof of Theorem~\ref{thm:buan-zhou-theorem} by using differential graded (=dg) algebras. We establish the result in the following three settings (Section~\ref{ss:special-case-algebra}):
\begin{itemize}
\item[(1)] the category $\Mod A$ of all modules, where $A$ is a ring;
\item[(2)] the category $\mod A$ of modules which are finitely generated over the base ring $k$, which is noetherian;
\item[(3)] the category $\fdmod A$ of modules which are finite-dimensional over the base ring $k$, which is a field.
\end{itemize}
Note that (3) is a special case of (2).  Let us briefly explain the idea. Let $A$ be a ring and $T$ be a 2-term silting complex in $K^b(\proj A)$.  Part (a) of our result in Setting (1) is the same as \cite[Corollary 4.3.3(b)(c)]{BreazModoi20} and a similar idea as that used for its proof was used in \cite[Section 4.3]{BreazModoi20}. Consider the truncated dg endomorphism algebra $\Gamma$ of $T$ (the dg endomorphism algebra of $T$, truncated at degree $0$), which is a non-positive dg algebra, and put $B=H^{0}(\Gamma)=\End_{K^b(\proj A)}(T)$. Then the derived Hom-functor associated with $T$ is a triangle equivalence between the derived categories of $A$ and $\Gamma$. Then (a) follows immediately, just as the classical tilting theorem of  Brenner and Bulter is a consequence of the derived equivalence between $A$ and $B$ together with the bijection between torsion pairs and intermediate $t$-structures (see \cite[Chapter III, Theorem 3.2]{Happel88} and \cite[Chapter I, Theorem 3.1]{BeligiannisReiten07}). 
The idea for proving (c) and (d) was already explained by Buan and Zhou in the paragraph after \cite[Theorem 1.1]{BuanZhou16}. Let us repeat here. Let $p\colon \Gamma\rightarrow B=H^0(\Gamma)$ be the canonical projection. Then by \cite[Proposition A.3]{BruestleYang13}, the induction functor $p_*\colon \per(\Gamma)\rightarrow \per(B)\simeq K^b(\proj B)$ takes $\Sigma A$ to a 2-term silting complex in $K^b(\proj B)$. This 2-term silting complex is exactly the complex $S$ in (b) and the algebra homomorphism $\pi\colon A\to\bar{A}$ in (d) is the one induced by the functor $p_*$. The results (f) and (g) follow from the adjunction between $p_*$ and the restriction functor $p^*$. Actually we first establish these results for non-positive dg algebras in Sections~\ref{ss:silting-theorem-first-step}~and~\ref{ss:silting-theroem-second-step} and then specialise it to algebras in Section~\ref{ss:special-case-algebra}.

We also study the question when $S$ is a tilting complex. We show that this is the case in the following three cases:
\begin{itemize}
\item[-] $A$ is hereditary, 
\item[-] $k$ is a field and $A$ is a finite-dimensional symmetric $k$-algebra, 
\item[-] $T$ is a left mutation of $A$ or a right mutation of $\Sigma A$.
\end{itemize}

The paper is structured as follows. In Section~\ref{s:non-positive-dg-algebras} we recall the basics on silting objects, $t$-structures and derived categories of dg algebras, including $t$-structures induced by silting objects and restrictions of such $t$-structures to certain subcategories. In Section~\ref{s:2-term-silting-objects} we study specifically $t$-structures induced by $2$-term silting objects. In Section~\ref{s:the-silting-theorem} we prove the silting theorem. In Section~\ref{s:stability} we study the question when $S$ is a tilting complex.

\medskip
\noindent\emph{Notations and conventions.} 
For a ring $A$, let $\Mod A$ denote the category of (right) $A$-modules, and $\mod A$ denote  its full subcategory of $A$-modules finite generated over $k$ and $\proj A$ denote its full subcategory of finitely generated projective modules, respectively. If $k$ is a field and $A$ is a $k$-algebra, we denote by $\fdmod A$ the category of $A$-modules which are finite-dimensional over $k$. For an additive category $\ca$, denote by $K^{b}(\ca)$ the bounded homotopy category of $\ca$ and, if $\ca$ is abelian, by $\cd^{b}(\ca)$ the bounded derived category of $\ca$. For a complex $X$ and an integer $i$, denote by $H^{i}(X)$ the $i$-th cohomology of $X$. 

Let $\ct$ be a triangulated category with shift functor $\Sigma$. For two subcategories $\cx$ and $\cy$ of $\ct$ we denote by $\cx*\cy$ the full subcategory of $\ct$ consisting of objects $Z$ such that there is a triangle $X\to Z\to Y\to \Sigma X$ with $X\in\cx$ and $Y\in\cy$.

Throughout let $k$ be a commutative ring. 

\medskip
\noindent\emph{Acknowledgement}. {The authors would like to thank Bernhard Keller and Xiao-Wu Chen for answering their questions. They are grateful to Yu Zhou for helpful comments and to Simion Breaz for pointing out the generalisation of Theorem~\ref{thm:buan-zhou-theorem}(a) with the references \cite{BreazModoi20,BreazModoi17}. Zongzhen Xie acknowledges support by the National Natural Science Foundation of China No. 12101320. Dong Yang acknowledges support by the National Natural Science Foundation of China No. 12031007. Houjun Zhang acknowledges support by a project funded by China Postdoctoral Science Foundation (2020M681540).}

\section{Non-positive dg algebras} 
\label{s:non-positive-dg-algebras}

In this section, we recall some notions and results on $t$-structures, silting objects and derived categories of differential graded (=dg) algebras.

\subsection{Silting objects and $t$-structures} 

Let $\mathcal{T}$ be a triangulated category with shift functor $\Sigma$. For a subcategory or a set of objects $\cs$ of $\mathcal{T}$, we denote by $\thick(\cs)$ the \emph{thick subcategory} of $\mathcal{T}$ generated by $\cs$, \ie the smallest triangulated subcategory of $\mathcal{T}$ which contains $\cs$ and which is closed under taking direct summands. In the following, we give the definitions of silting objects and $t$-structures. For more details we refer to \cite{AiharaIyama12,KellerVossieck88,BeilinsonBernsteinDeligne82}.
\begin{definition}
\label{def:silting} 
An object $M$ of $\mathcal{T}$ is said to be 

\begin{itemize}
\item[(1)]
\emph{presilting} if $\Hom_{\mathcal{T}}(M,\Sigma^i M)=0$ for all $i>0$;

\item[(2)]
\emph{silting} if it is presilting and $\mathcal{T}=\thick(M)$;

\item[(3)]
\emph{tilting} if $\Hom_{\mathcal{T}}(M,\Sigma^i M)=0$ for $i\neq 0$ and $\mathcal{T}=\thick(M)$.
\end{itemize}
\end{definition}

The following result is \cite[Theorem 5.5]{MendozaSaenzSantiagoSouto13} (see also \cite[Proposition 2.8]{IyamaYang18}).

\begin{proposition}
\label{prop:silting-->co-t-str}
Let $M$ be a silting object of $\ct$. Then
\[
\ct=\bigcup_{n\in\mathbb{N}}(\Sigma^{-n}\add(M)*\cdots*\Sigma^n\add(M)).
\]
\end{proposition}

Let $M$ be a silting object of $\ct$ and $n$ be a positive integer. We say that a silting object $N$ of $\ct$ is \emph{$n$-term} with respect to $M$ if $N\in\add(M)*\cdots*\Sigma^{n-1}\add(M)$. The following lemma is a reformulation of \cite[Proposition 3.8]{AdachiMizunoYang19}. One can adapt the proof of \cite[Corollary 2.4]{IyamaJorgensenYang14} to give a direct proof.

\begin{lemma}
\label{lem:n-term-silting}
Let $M$ and $N$ be silting objects of $\ct$. Then $N$ is $n$-term with respect to $M$ if and only if $\Sigma^{n-1}M$ is $n$-term with respect to $N$.
\end{lemma}

Let $A$ be a $k$-algebra. Then the free module $A_A$ of rank $1$ is a typical silting object of $K^b(\proj A)$. It is in fact a tilting object. We will call silting objects (respectively, tilting objects) of $K^b(\proj A)$ \emph{silting complexes} (respectively, \emph{tilting complexes}), and a silting complex is said to be \emph{$n$-term} if it has non-trivial components only in degrees $-(n-1),\ldots,-1,0$.

\medskip

\begin{definition}\label{2.2}
A $\emph t$-$\emph structure$ on $\mathcal{T}$ is a pair $(\mathcal{T}^{\leq 0},\mathcal{T}^{\geq 0})$ of strict (that is, closed under isomorphisms) and full subcategories of $\mathcal{T}$ such that the following conditions are satisfied for $\mathcal{T}^{\geq i}:=\Sigma^{-i}\mathcal{T}^{\geq 0}$, $\mathcal{T}^{\leq i}:=\Sigma^{-i}\mathcal{T}^{\leq 0}$
\begin{itemize}
\item[(1)] $\mathcal{T}^{\leq 0}\subseteq \mathcal{T}^{\leq 1}$ and $\mathcal{T}^{\geq 1}\subseteq \mathcal{T}^{\geq 0}$;
\item[(2)] $\Hom_{\mathcal{T}}(X,Y)=0$ for $X\in \mathcal{T}^{\leq 0}$ and $Y\in \mathcal{T}^{\geq 1}$;
\item[(3)] $\mathcal{T}=\mathcal{T}^{\leq 0}\ast\mathcal{T}^{\geq 1}$.
\end{itemize}
The $t$-structure $(\mathcal{T}^{\leq 0},\mathcal{T}^{\geq 0})$ is said to be \emph{bounded} if
\[
\bigcup_{i\in \mathbb{Z}}\mathcal{T}^{\leq i}=\mathcal{T}=\bigcup_{i\in \mathbb{Z}}\mathcal{T}^{\geq i}.
\]
\end{definition}

The category $\ca=\mathcal{T}^{\leq 0}\cap \mathcal{T}^{\geq 0}$ is called the \emph{heart} of the $t$-structure $(\mathcal{T}^{\leq 0},\mathcal{T}^{\geq 0})$ and it is an abelian category due to \cite[Th\'eor\`eme 1.3.6]{BeilinsonBernsteinDeligne82}. For any $Z\in\ct$, there is a triangle by the condition (3)
\[
\xymatrix{
X\ar[r] & Z\ar[r] & Y\ar[r] & \Sigma X
}
\]
with $X\in\ct^{\leq 0}$ and $Y\in\ct^{\geq 1}$.
Due to (1) and (2) this triangle is unique up to a unique isomorphism, so the correspondences $Z\mapsto X$ and $Z\mapsto Y$ extend to functors $\sigma^{\leq0}\colon \mathcal{T}\rightarrow \mathcal{T}^{\leq 0}$ and $\sigma^{\geq 1}\colon \mathcal{T}\rightarrow \mathcal{T}^{\geq 1}$, respectively, called the $\emph truncation~~functors$. The functor $\sigma^{\leq0}$ is right adjoint to the inclusion $\mathcal{T}^{\leq 0}\rightarrow \mathcal{T}$ and the functor $\sigma^{\geq 1}$ is left adjoint to the inclusion $\mathcal{T}^{\geq 1}\rightarrow \mathcal{T}$. Moreover, the functor $H^0=\sigma^{\leq 0}\sigma^{\geq 0}\colon\ct\to\ca$ is a cohomological functor.

\begin{example}
Let $\ca$ be an abelian category.
Then $\mathcal{D}^{b}(\mathcal{A})$ has a $t$-structure $(\mathcal{D}^{\leq 0}_{\mathrm{std}}, \mathcal{D}^{\geq 0}_{\mathrm{std}})$, where
\begin{align*}
\mathcal{D}^{\leq 0}_{\mathrm{std}}&=\{X\in \mathcal{D}^{b}(\mathcal{A})\mid H^{i}(X)=0 ~~\forall i>0\},\\
\mathcal{D}^{\geq 0}_{\mathrm{std}}&=\{X\in \mathcal{D}^{b}(\mathcal{A})\mid H^{i}(X)=0 ~~\forall i<0\}.
\end{align*}
This $t$-structure is bounded and is called the \emph{standard $t$-structure} on $\cd^b(\ca)$.
\end{example}

Let $(\mathcal{T}^{\leq 0},\mathcal{T}^{\geq 0})$ be a $t$-structure  and $n$ be a positive integer. A $t$-structure $(\mathcal{T}^{\prime\leq 0},\mathcal{T}^{\prime\geq 0})$ is said to be \emph{$n$-intermediate} with respect to $(\mathcal{T}^{\leq 0},\mathcal{T}^{\geq 0})$ if $\mathcal{T}^{\leq -n+1}\subseteq\mathcal{T}^{\prime\leq 0}\subseteq\mathcal{T}^{\leq 0}$, or equivalently, $\mathcal{T}^{\geq 0}\subseteq\mathcal{T}^{\prime\geq 0}\subseteq\mathcal{T}^{\geq -n+1}$. In this case, the $t$-structure $(\mathcal{T}^{\leq 0},\mathcal{T}^{\geq 0})$ is bounded if and only if so is $(\mathcal{T}^{\prime\leq 0},\mathcal{T}^{\prime\geq 0})$. When $n=2$, we say \emph{intermediate} instead of $2$-intermediate. 
The following construction is due to Happel--Reiten--Smal{\o}, see \cite[Chapter 1, Proposition 2.1]{HappelReitenSmaloe96} and \cite[Proposition 2.5]{Bridgeland05}. 

\begin{construction}
\label{constr:HRS-tilt}
Let $(\ct^{\leq 0},\ct^{\geq 0})$ be a bounded $t$-structure on $\ct$ with heart $\ca$ and let $(\mathscr{T},\mathscr{F})$ be a torsion pair on $\ca$. Define two full subcategories of $\ct$ by
\begin{align*}
\ct^{\prime\leq 0}&=\{X\in\ct\mid H^i(X)=0 \text{ for } i>0,~\text{and}~ H^0(X)\in\mathscr{T}\},\\
\ct^{\prime\geq 0}&=\{X\in\ct\mid H^i(X)=0 \text{ for } i<-1,~\text{and}~H^{-1}(X)\in\mathscr{F}\}.
\end{align*}
Then $(\mathcal{T}^{\prime\leq 0}, \mathcal{T}^{\prime\geq 0})$ is a $t$-structure on $\ct$ with heart $\cb=(\Sigma\mathscr{F})\ast \mathscr{T}$, and $(\Sigma\mathscr{F},\mathscr{T})$ is a torsion pair of $\cb$.
It is clearly intermediate with respect to the given $t$-structure $(\ct^{\leq 0},\ct^{\geq 0})$. It is called the \emph{Happel--Reiten--Smal\o~  
tilt} of $(\ct^{\leq 0},\ct^{\geq 0})$ at the torsion pair $(\mathscr{T},\mathscr{F})$.
\end{construction}

\begin{proposition}[{\cite[Chapter I, Theorem 3.1]{BeligiannisReiten07}, \cite[Proposition 2.1]{Woolf10}}]
\label{prop:bijection-between-int-t-str-and-torsion-pairs}
Let $(\mathcal{T}^{\leq 0},\mathcal{T}^{\geq 0})$ be a bounded $t$-structure on $\mathcal{T}$ with heart $\mathcal{A}$. Then taking the HRS tilt defines a bijective map from the set of torsion pairs of $\mathcal{A}$ to the set of intermediate $t$-structures on $\ct$ with respect to $(\mathcal{T}^{\leq 0},\mathcal{T}^{\geq 0})$. Its inverse takes $(\mathcal{T}^{\prime \leq 0},\mathcal{T}^{\prime \geq 0})$ to $(\mathcal{T}^{\prime \leq 0}\cap \mathcal{A},\mathcal{T}^{\prime \geq 1}\cap \mathcal{A})$.
\end{proposition}

\subsection{Derived categories of dg algebras}
A \emph{dg $k$-algebra} $A$ is a graded $k$-algebra with a differential $d$ which makes $A$ a complex of $k$-modules such that the multiplication $A \otimes_{k} A \rightarrow A$ is a chain map, \ie the graded Leibniz rule holds:
\[
d(ab) = d(a)b + (-1)^{|a|}ad(b),
\]
where $a$ and $b$ are homogeneous elements of $A$ and $|a|$ denotes the degree of $a$. For two dg $k$-algebras $A$
and $B$, a dg algebra homomorphism $f\colon A \rightarrow B $ is a $k$-algebra homomorphism which is a chain map.

We follow \cite{Keller94,Keller06d} to introduce the derived category of a dg algebra. Let $A$ be a dg $k$-algebra. For two dg $A$-modules $M$ and $N$, the morphism complex $\cHom_{A}(M,N)$ is the complex  whose degree $i$ component consists of the homogeneous $A$-module homomorphisms from $M$ to $N$ of degree $i$, and whose differential is given by
\[
d(f) = d_{N}\circ f- (-1)^{|f|}f \circ d_{M},
\]
where $f \in \cHom_{A}(M, N)$ is homogeneous. We call $\cEnd_A(M):=\cHom_A(M,M)$ the \emph{dg endomorphism algebra} of $M$. Let $\mathcal{C}_{dg}(A)$ be the dg category of dg $A$-modules: its objects are dg $A$-modules, and for dg $A$-modules $M$ and $N$ we have $\Hom_{\mathcal{C}_{dg}(A)}(M,N)=\cHom_A(M,N)$. Let $K(A)=H^{0}\mathcal{C}_{dg}(A)$ be the homotopy category of $\mathcal{C}_{dg}(A)$: its objects are dg $A$-modules, and for two dg $A$-modules $M$ and $N$ we have
\[
\Hom_{K(A)}(M,N):=H^{0}\Hom_{\mathcal{C}_{dg}(A)}(M,N)=H^0\cHom_A(M,N).
\]
The \emph{derived category} $\mathcal{D}(A)$ of dg $A$-modules is defined as the triangle quotient of $K(A)$ by the subcategory of acyclic dg $A$-modules. A dg $A$-module $P$ is said to be \emph{$\ck$-projective} if $\cHom_A(P,M)$ is acyclic whenever $M$ is an acyclic dg $A$-module.  If $P$ is a $\ck$-projective dg $A$-module and $M$ is any dg $A$-module, then
\begin{align}\label{eq:morphisms-in-derived-category}
\Hom_{\cd(A)}(P,\Sigma^i M)=\Hom_{K(A)}(P,\Sigma^i M)
\end{align}
for any $i\in\mathbb{Z}$.
For example, $A_A$ is $\ck$-projective. Thus we have
\begin{align}\label{eq:morphisms-vs-cohomologies}
\Hom_{\cd(A)}(A,\Sigma^i M)=\Hom_{K(A)}(A,\Sigma^i M)=H^i(M).
\end{align}
In particular, by Auslander's projectivisation, there is a $k$-linear equivalence
\begin{align}
\label{eq:projectivisation}
H^0=\Hom_{\cd(A)}(A,?)\colon \add_{\cd(A)}(A_A)\to \proj H^0(A).
\end{align}
By \cite[Theorem 3.1]{Keller94}, for any dg $A$-module $M$ there is a $\ck$-projective dg module $P_M$ together with a quasi-isomorphism $P_M\to M$ of dg $A$-modules.

\medskip

Denote by $\cd^{b}(A)$ the full subcategory of $\mathcal{D}(A)$ whose objects are the dg $A$-modules with bounded total cohomology. Let $\per(A) = \thick_{\mathcal{D}(A)}(A_{A})$ be the
thick subcategory of $\mathcal{D}(A)$ generated by $A_{A}$. In the case $A$ is a $k$-algebra (viewed as a dg $k$-algebra concentrated in degree $0$), a dg $A$-module is exactly a complex of $A$-modules, we have $\mathcal{D}(A)=\mathcal{D}(\Mod A)$, and there are canonical equivalences $K^{b}(\proj A)\simeq\per(A)$ and $\mathcal{D}^{b}(\Mod A)\simeq\mathcal{D}^{b}(A)$. The first restricted equivalence in the following lemma is well-known, since $\per(A)$ is exactly the subcategory of $\cd(A)$ of compact objects (\cite[Theorem 5.3]{Keller94}).

\begin{lemma}
\label{lem:restriction-of-triangle-equivalence-large-module-level}
Let $A$ and $B$ be two dg $k$-algebras and $F\colon\cd(B)\to\cd(A)$ be a triangle equivalence. Then $F$ restricts to triangle equivalences
 $\per(B)\rightarrow \per(A)$ and $\mathcal{D}^{b}(B)\rightarrow \mathcal{D}^{b}(A)$.
\end{lemma}
\begin{proof}
We claim that
\[
\cd^b(A)=\{M\in\cd(A)\mid\Hom_{\cd(A)}(P,\Sigma^i M)=0~~\forall i\in\mathbb{Z} \text{ with }|i|\gg 0~~\forall P\in\per(A)\}.
\]
Then the equivalence on $\cd^b$ follows from that on $\per$. 
In the above claim the inclusion `$\supseteq$' is obvious by \eqref{eq:morphisms-vs-cohomologies} since $A_A\in\per(A)$. For the inverse inclusion we apply d\'evissage. Let 
\[
\cs=\{P\in\cd(A)\mid \Hom_{\cd(A)}(P,\Sigma^i M)=0~~\forall i\in\mathbb{Z} \text{ with }|i|\gg 0~~\forall M\in\cd^b(A)\}.
\] 
Again by \eqref{eq:morphisms-vs-cohomologies} we have $A_A\in\cs$. Moreover, it is easily seen that $\cs$ is closed under taking shifts, extensions and direct summands. Therefore $\per(A)\subseteq\cs$ and the inverse inclusion follows.
\end{proof}

\subsection{Derived equivalences}
Let $A$ and $B$ be two dg $k$-algebras and $X$ be a dg $B$-$A$-bimodule. Then there is an adjoint pair of triangle functors
\[
\xymatrix@C=5pc{
\cd(A)\ar@<-.7ex>[r]_{\RHom_A(X,?)} & \cd(B).\ar@<-.7ex>[l]_{?\lten_B X}
}
\]
Also there is a dg $k$-algebra homomorphism $\rho\colon B\rightarrow \cEnd_{A}(X)$, $b\mapsto (x\mapsto bx)$, which we call the \emph{representation map} of $X$. The following is the one-object version of \cite[Lemma 6.1 (a)]{Keller94}.

\begin{lemma}
\label{lem:when-tensor-is-equivalence} 
Suppose that $X_A$ is $\ck$-projective over $A$\footnote{In \cite[Lemma 6.1 (a) and Lemma 6.2]{Keller94}, the assumption on $X_A$ is to have the property (P) (\ie to be cofibrant in the terminology of \cite[Section 3.2]{Keller06d}). But the proofs still work if we make the slightly weaker assumption that $X_A$ is $\ck$-projective.}. Then the derived tensor functor $?\lten_B X\colon \mathcal{D}(B)\rightarrow \mathcal{D}(A)$ is fully faithful if and only if $X_A$ is self-compact in $\cd(A)$ and $\rho$ is a quasi-isomorphism; it is a triangle equivalence if and only if $X_A$ is a compact generator of $\mathcal{D}(A)$ {\rm (}\ie $\thick_{\cd(A)}(X_A)=\per(A)${\rm )} and $\rho$ is a quasi-isomorphism. 
\end{lemma}

\begin{corollary}
Let $M$ be a $\ck$-projective dg $A$-module which is a compact generator of $\cd(A)$. Then $?\lten_{\cEnd_A(M)}M\colon\cd(\cEnd_A(M))\to \cd(A)$ is a triangle equivalence.
\end{corollary}
\begin{proof}
This follows immediately from Lemma~\ref{lem:when-tensor-is-equivalence} because in this case the representation map is the identity map of $\cEnd_A(M)$.
\end{proof}

The following is the one-object version of \cite[Lemma 6.2 (b)]{Keller94}.

\begin{lemma}
\label{lem:turning-Hom-to-tensor}
Assume that $X_A$ is $\ck$-projective over $A$. If $?\lten_B X\colon\cd(B)\to\cd(A)$ is a triangle equivalence, then there is a natural isomorphism of triangle functors $\RHom_A(X,?)\cong ~?\lten_A X^{tr}$, where $X^{tr}=\cHom_A(X,A)$.
\end{lemma}

Let $A$ and $B$ be two dg $k$-algebras and  $f\colon B \to A$ be a homomorphism of dg algebras. Then we can view $A$ as a dg $B$-$A$-bimodule via $f$, whose representation map is exactly $f$. There is the induction functor $f_*=~?\lten_B A\colon\cd(B)\to\cd(A)$ and the restriction functor $f^*=\RHom_{A}(A,?)\colon\cd(A)\to\cd(B)$, which form an adjoint pair of triangle functors. Moreover, the following diagram commutes
\begin{align}
\label{eq:inducting-free-modules}
\xymatrix{
\add_{\cd(B)}(B_B)\ar[r]^{H^0} \ar[d]^{f_*} & \proj H^0(B)\ar[d]^{(H^0f)_*}\\
\add_{\cd(A)}(A_A)\ar[r]^{H^0} & \proj H^0(A).
}
\end{align}

\begin{corollary}
\label{cor:when-pushing-out-is-an-equivalence}
$f_*\colon\cd(B)\to\cd(A)$ is a triangle equivalence if and only if $f$ is a quasi-isomorphism.
\end{corollary}
\begin{proof}
This follows immediately from Lemma~\ref{lem:when-tensor-is-equivalence} because in this case the representation map is exactly $f$.
\end{proof}

\subsubsection{The truncated dg endomorphism algebra of a silting object} 
\label{sss:truncated-dg-endo-algebra-of-a-silting}
A dg $k$-algebra is said to be \emph{non-positive} if its components in positive degrees vanish. Such dg algebras are closely related to silting objects. Let $A$ be a dg $k$-algebra and $T$ be a silting object in $\per(A)$. For example, $A_A$ is a silting object in $\per(A)$ if $A$ is non-positive.  We assume further that $T$ is a $\ck$-projective dg $A$-module and construct a non-positive dg algebra from $T$ as follows. Recall that the degree $i$ component of $\cEnd_A(T)=\cHom_{A}(T,T)$ consists of those $A$-linear maps from $T$ to itself which are homogeneous maps of degree $i$. The differential of $\cEnd_A(T)$ takes a homogeneous map $f$ of degree $i$ to $d_T\circ f-(-1)^{i}f\circ d_T$. This differential and the composition of maps make $\cEnd_A(T)$ into a dg $k$-algebra. Since $T$ is a silting object, it follows from \eqref{eq:morphisms-in-derived-category} that $\cEnd_A(T)$ has cohomologies concentrated in non-positive degrees and $H^0\cEnd_A(T)=\End_{\cd(A)}(T)$. Therefore $B:=\sigma^{\leq 0}\cEnd_A(T)$ is a non-positive dg $k$-algebra and the embedding $B\to\cEnd_A(T)$ is a quasi-isomorphism of dg algebras, where $\sigma^{\leq 0}$ is the standard truncation at degree $0$. In this way $T$ becomes a dg $B$-$A$-bimodule with representation map the embedding $B\to\cEnd_A(T)$, so by Lemma~\ref{lem:when-tensor-is-equivalence}, the associated derived tensor functor $?\lten_{B}T\colon\cd(B)\to\cd(A)$ is a triangle equivalence with quasi-inverse $\RHom_A(T,?)=\cHom_A(T,?)\colon\cd(A)\to\cd(B)$. We will call this dg algebra $B$ the \emph{truncated dg endomorphism algebra} of $T$.

\smallskip
We assume further that $A$ is non-positive and that $T$ is \emph{$n$-term}, that is, it is $n$-term with respect to $A_A$. Let $\widetilde{S}=\cHom_A(T,\Sigma^{n-1} A)$. Then by Lemma~\ref{lem:n-term-silting}, $\Sigma^{n-1} A$ is a silting object of $\per(A)$ which is $n$-term with respect to $T$, and hence $\widetilde{S}$ is an $n$-term silting object of $\per(B)$. Moreover, $\widetilde{S}$ is naturally a dg $A$-$B$-bimodule, and by Lemma~\ref{lem:turning-Hom-to-tensor} we have isomorphisms of triangle equivalences:  
\[
\RHom_A(T,?)\cong \Sigma^{1-n}(?\lten_A \widetilde{S})~\text{and}~
?\lten_B T\cong\Sigma^{n-1}\RHom_B(\widetilde{S},?).
\] 

\subsection{Non-positive dg algebras: silting $t$-structures}
\label{ss:non-positive-dg-algebra}

Recall that a dg $k$-algebra $A$ is said to be non-positive if the degree $i$ component $A^{i}$ vanishes for all $i>0$. A $k$-algebra can be viewed as a non-positive dg $k$-algebra concentrated in degree $0$.

Let $A$ be a non-positive dg $k$-algebra. 
The canonical projection $p\colon A\rightarrow H^{0}(A)$ is a homomorphism of dg algebras. We view a module over $H^{0}(A)$ as a dg module over $A$ via this homomorphism. This defines a natural functor $\Mod H^{0}(A)\rightarrow \mathcal{D}(A)$, which is fully faithful. We will identify $\Mod H^{0}(A)$ with its essential image in $\mathcal{D}(A)$. Put
\begin{align*}
\mathcal{D}_\std^{\leq 0}&=\{X\in \mathcal{D}(A)\mid H^i(X)=0 ~\forall i>0\},\\
\mathcal{D}_\std^{\geq 0}&=\{X\in \mathcal{D}(A)\mid H^i(X)=0 ~\forall i<0\}.
\end{align*}

\begin{lemma}[{\cite[Lemma 2.2 and Proposition 2.3]{Amiot09}}]
\label{lem:standard-t-structure-on-unbounded-level}
The pair $(\mathcal{D}_\std^{\leq 0},\mathcal{D}_\std^{\geq 0})$ is a $t$-structure on $\mathcal{D}(A)$ and the functor $H^{0}\colon \mathcal{D}(A)\rightarrow \Mod H^{0}(A)$ of taking the 0-th cohomology restricts to an equivalence from the heart to $\Mod H^{0}(A)$, with the embedding $\Mod H^0(A)\to\cd(A)$ as a quasi-inverse. The associated truncation functors are the standard truncations.
\end{lemma}

Now let $T$ be a silting object in $\per(A)$ and put
\begin{align*}
\mathcal{D}_T^{\leq 0}&=\{X\in \mathcal{D}(A)\mid \Hom_{\cd(A)}(T,\Sigma^i X)=0 ~\forall i>0\},\\
\mathcal{D}_T^{\geq 0}&=\{X\in \mathcal{D}(A)\mid \Hom_{\cd(A)}(T,\Sigma^i X)=0 ~\forall i<0\}.
\end{align*}
Thanks to \eqref{eq:morphisms-vs-cohomologies}, we have $\cd_\std^{\leq 0}=\cd_A^{\leq 0}$ and $\cd_\std^{\geq 0}=\cd_A^{\geq 0}$. The first part of the following lemma is a special case of \cite[Theorem 1.3]{HoshinoKatoMiyachi02} (\confer also \cite[Chapter III, Proposition 2.8]{BeligiannisReiten07} and \cite[Corollary 4.7]{AiharaIyama12}). We include a direct proof.

\begin{lemma}
\label{lem:silting-t-structure-unbounded-level}
The pair $(\cd_T^{\leq 0},\cd_T^{\geq 0})$ is a $t$-structure on $\cd(A)$ and the Hom-functor $\Hom_{\cd(A)}(T,?)\colon\cd(A)\to \Mod\End_{\cd(A)}(T)$ restricts to an equivalence from the heart to $\Mod \End_{\cd(A)}(T)$. Moreover, it is $n$-intermediate with respect to the standard $t$-structure if and only if $T$ is $n$-term.
\end{lemma}
\begin{proof}
Assume that $T$ is $\ck$-projective over $A$.
Let $B$ be the truncated dg endomorphism algebra of $T$ and recall from Section~\ref{sss:truncated-dg-endo-algebra-of-a-silting} that the derived tensor functor $?\lten_{B}T\colon\cd(B)\to\cd(A)$ is a triangle equivalence. It takes $B$ to $T$ and it takes the pair $(\cd_\std^{\leq 0}(B),\cd_\std^{\geq 0}(B))$ to the pair $(\cd_T^{\leq 0},\cd_T^{\geq 0})$. The former pair, by Lemma~\ref{lem:standard-t-structure-on-unbounded-level}, is a $t$-structure on $\cd(B)$ such that $H^0=\Hom_{\cd(B)}(B,?)\colon\cd(B)\to\Mod\End_{\cd(A)}(T)$ restricts to an equivalence from the heart to $\Mod\End_{\cd(A)}(T)$. The first statement follows immediately.

For the second statement, it suffices to show that if $T$ is $n$-term but not $(n-1)$-term then $(\cd_T^{\leq 0},\cd_T^{\geq 0})$ is $n$-intermediate but not $(n-1)$-intermediate. Assume that $T$ is $n$-term. Then by d\'evissage we have $\cd_\std^{\leq -n+1}\subseteq\cd_T^{\leq 0}$ and $\cd_T^{\geq 0}\supseteq\cd_\std^{\geq 0}$. This shows that $(\cd_T^{\leq 0},\cd_T^{\geq 0})$ is $n$-intermediate. Assume further that $T$ is not $(n-1)$-term. Then there is a triangle
\[
\xymatrix{
T'\ar[r] & T\ar[r] & T''\ar[r] & \Sigma T'
}
\]
with $T'\in\add(A)*\cdots*\Sigma^{n-2}\add(A)$ and $0\not\cong T''\in\Sigma^{n-1}\add(A)$. Since $H^0(\Sigma^{-n+1}T'')$ is a non-zero projective $H^0(A)$-module, there exists an $H^0(A)$-module $Y$ such that $\Hom_{H^0(A)}(H^0(\Sigma^{-n+1}T''),Y)\neq 0$. We let $X=\Sigma^{n-2}Y$ and apply $\Hom_{\cd(A)}(?,X)$ to the above triangle. Inspection on the resulting long sequence shows that 
\begin{align*}
\Hom_{\cd(A)}(T,\Sigma X)&\cong\Hom_{\cd(A)}(T'',\Sigma X)=\Hom_{\cd(A)}(\Sigma^{-n+1}T'',Y)\\
&=\Hom_{H^0(A)}(H^0(\Sigma^{-n+1}T''),Y)\neq 0,
\end{align*} 
so $X\not\in\cd_T^{\leq 0}$. This shows that $\cd_{\std}^{\leq -n+2}\not\subseteq\cd_T^{\leq 0}$, and hence $(\cd_T^{\leq 0},\cd_T^{\geq 0})$ is not $(n-1)$-intermediate.
\end{proof}

\subsection{Restriction of silting $t$-structures on subcategories}
\label{ss:restriction-of-silting-t-structures}
Let $A$ be a non-positive dg $k$-algebra. Let $\mathcal{D}_{fg}(A)$ denote the full subcategory of $\mathcal{D}(A)$ consisting of dg $A$-modules $M$
such that $H^{\ast}(M)$ is finitely generated over the base ring $k$. If $k$ is a noetherian ring, this is a triangulated subcategory of $\mathcal{D}(A)$; if further $A$ is a $k$-algebra and is finitely generated over $k$, then there is a canonical triangle equivalence $\mathcal{D}^{b}(\mod A)\rightarrow\mathcal{D}_{fg}(A)$. When $k$ is a field, let $\mathcal{D}_{fd}(A)$ denote the full subcategory of $\mathcal{D}(A)$ consisting of those dg $A$-modules with finite-dimensional total cohomology. So in this case $\cd_{fd}(A)=\cd_{fg}(A)$. If $A$ is a finite-dimensional $k$-algebra over the field $k$, then there is a canonical triangle equivalence $\mathcal{D}^{b}(\fdmod A)\to\mathcal{D}_{fd}(A)$.

Let $T$ be a silting object of $\per(A)$. Throughout this subsection we fix one of the following settings:
\begin{itemize}
\item[(1)] $\mathbf{D}=\cd^b$ and $\mathscr{M}=\Mod$;
\item[(2)] $k$ is a noetherian ring, $\mathbf{D}=\cd_{fg}$ and $\mathscr{M}=\mod$;
\item[(3)] $k$ is a field, $\mathbf{D}=\cd_{fd}$ and $\mathscr{M}=\fdmod$.
\end{itemize}
Note that (3) is a special case of (2).
Put 
\begin{align*}
\mathbf{D}_\std^{\leq 0}=\mathbf{D}(A)\cap\cd_\std^{\leq 0},~~ \mathbf{D}_\std^{\geq 0}=\mathbf{D}(A)\cap\cd_\std^{\geq 0};\\
\mathbf{D}_T^{\leq 0}=\mathbf{D}(A)\cap\cd_T^{\leq 0},~~
\mathbf{D}_T^{\geq 0}=\mathbf{D}(A)\cap\cd_T^{\geq 0}.
\end{align*}

\begin{proposition}
\label{prop:restriction-of-silting-t-structure-on-subcategories}

\begin{itemize}
\item[(a)]
Let $B$ be another non-positive dg $k$-algebra. Then any triangle equivalence $F\colon\cd(B)\to\cd(A)$ restricts to a triangle equivalence $\mathbf{D}(B)\to\mathbf{D}(A)$.
\item[(b)] The pair $(\mathbf{D}_\std^{\leq 0},\mathbf{D}_\std^{\geq 0})$ is a bounded $t$-structure on $\mathbf{D}(A)$ and $H^{0}\colon \mathbf{D}(A)\rightarrow \mathscr{M}_{H^0(A)}$, the functor of taking the 0-th cohomology, restricts to an equivalence from the heart to $\mathscr{M}_{H^0(A)}$, with the embedding $\mathscr{M}_{H^0(A)}\to\mathbf{D}(A)$ as a quasi-inverse. The associated truncation functors are the standard truncation functors.
\item[(c)] The pair $(\mathbf{D}_T^{\leq 0},\mathbf{D}_T^{\geq 0})$  is a bounded $t$-structure on $\mathbf{D}(A)$ and the Hom-functor $\Hom_{\mathcal{D}(A)}(T,?)\colon\mathbf{D}(A)\to\mathscr{M}_{\End_{\cd(A)}(T)}$ restricts to an equivalence from the heart to $\mathscr{M}_{\End_{\cd(A)}(T)}$. It is $n$-intermediate with respect to the standard $t$-structure if and only if $T$ is $n$-term.
\end{itemize}
\end{proposition}
\begin{proof}
(a) For the setting (1), this is contained in Lemma~\ref{lem:restriction-of-triangle-equivalence-large-module-level}. For the setting (3), this is contained in \cite[Lemma 3.1]{KalckYang18a}.
For the setting (2), we claim that $\cd_{fg}(A)$ admits the following intrinsic description
\begin{align*}
\cd_{fg}(A)&=\{X\in\cd(A)\mid \bigoplus_{i\in\mathbb{Z}}\Hom_{\cd(A)}(P,\Sigma^i X) \text{ is finitely generated over $k$ }\\
&\qquad\qquad\qquad\qquad \forall P\in\per(A)\}.
\end{align*}
This follows by d\'evissage, similar to that used in the proof of Lemma~\ref{lem:restriction-of-triangle-equivalence-large-module-level}. Then the desired result follows from the restriction of $F$ on $\per$ (Lemma~\ref{lem:restriction-of-triangle-equivalence-large-module-level}).

(b) Since the standard truncation $\sigma^{\leq 0}$ preserves cohomologies in non-positive degrees, it restricts to a functor $\sigma^{\leq 0}\colon \mathbf{D}(A)\to\mathbf{D}(A)$. It follows that the $t$-structure in Lemma~\ref{lem:standard-t-structure-on-unbounded-level} restricts to $\mathbf{D}(A)$ and produces the $t$-structure $(\mathbf{D}_\std^{\leq 0},\mathbf{D}_\std^{\geq 0})$. It is rather clear that this $t$-structure is bounded. We remark that in the setting (3) this statement is \cite[Proposition 2.1(b)]{KalckYang16}.

(c) Similar to Lemma~\ref{lem:silting-t-structure-unbounded-level}.
\end{proof}

\section{2-term silting objects on non-positive dg algebras} 
\label{s:2-term-silting-objects}

Let $A$ be a non-positive dg $k$-algebra. Throughout this section we fix one of the following settings:
\begin{itemize}
\item[(1)] $\mathbf{D}=\cd^b$ and $\mathscr{M}=\Mod$;
\item[(2)] $k$ is a noetherian ring, $\mathbf{D}=\cd_{fg}$ and $\mathscr{M}=\mod$;
\item[(3)] $k$ is a field, $\mathbf{D}=\cd_{fd}$ and $\mathscr{M}=\fdmod$.
\end{itemize}
Put
\begin{align*}
\mathbf{D}_\std^{\leq 0}=\mathbf{D}(A)\cap\cd_\std^{\leq 0},~~ \mathbf{D}_\std^{\geq 0}=\mathbf{D}(A)\cap\cd_\std^{\geq 0}.
\end{align*}
By Proposition~\ref{prop:restriction-of-silting-t-structure-on-subcategories}, the pair $(\mathbf{D}_\std^{\leq 0},\mathbf{D}_\std^{\geq 0})$ is a bounded $t$-structure on $\mathbf{D}(A)$.
Let $T$ be a 2-term silting object of $\per(A)$, which is $\ck$-projective over $A$. Put
\begin{align*}
\mathbf{D}_T^{\leq 0}=\mathbf{D}(A)\cap\cd_T^{\leq 0},~~
\mathbf{D}_T^{\geq 0}=\mathbf{D}(A)\cap\cd_T^{\geq 0}.
\end{align*}
By Proposition~\ref{prop:restriction-of-silting-t-structure-on-subcategories}, the pair $(\mathbf{D}_T^{\leq 0},\mathbf{D}_T^{\geq 0})$ is a bounded $t$-structure on $\mathbf{D}(A)$, which is intermediate with respect to $(\mathbf{D}_\std^{\leq 0},\mathbf{D}_\std^{\geq 0})$.

Consider the full subcategories of $\mathscr{M}_{H^0(A)}$ given by
\begin{align*}
\mathscr{T}_T&=\{M\in \mathscr{M}_{H^0(A)}\mid \Hom_{\mathbf{D}(A)}(T,\Sigma M)=0\},\\
\mathscr{F}_T&=\{N\in\mathscr{M}_{H^0(A)}\mid \Hom_{\mathbf{D}(A)}(T,N)=0\}.
\end{align*}
For $M\in\mathscr{M}_{H^0(A)}$ we have $\Hom_{\mathbf{D}(A)}(T,\Sigma^i M)=0$ unless $i=0,1$, since $T$ is $2$-term. 
It follows that $\mathscr{T}_T=\mathscr{M}_{H^0(A)}\cap\mathbf{D}^{\leq 0}_T$ and $\mathscr{F}_T=\mathscr{M}_{H^0(A)}\cap\mathbf{D}^{\geq 1}_T$. The following result is then a consequence of Proposition~\ref{prop:bijection-between-int-t-str-and-torsion-pairs}.

\begin{proposition}
\label{prop:torsion-pair-from-2-term-silting}
The pair $(\mathscr{T}_T,\mathscr{F}_T)$ is a torsion pair of $\mathscr{M}_{H^0(A)}$ and $(\mathbf{D}_T^{\leq 0},\mathbf{D}_T^{\geq 0})$ is the HRS tilt of the standard $t$-structure at this torsion pair.
\end{proposition}

Set
\begin{align*}
\mathrm{tors}(\mathscr{M}_{H^0(A)}) &=\text{the set of torsion pairs of}\ \mathscr{M}_{H^0(A)},\\
\mathrm{2}\text{-}\mathrm{silt}(A) &=\text{the set of equivalence classes of 2-term silting objects $T$ in}\ \per(A),\\
\mathrm{int}\text{-}\mathrm{t}\text{-}\mathrm{str}(\mathbf{D}(A)) &= \text{the set of intermediate bounded $t$-structures on}\ \mathbf{D}(A)\ \text{with respect} \\
   & \text{\ \ \ \  to the standard $t$-structure}\ (\mathbf{D}_\std^{\leq 0},\mathbf{D}_\std^{\geq 0}).
\end{align*}
Here two silting objects $T$ and $S$ of $\per(A)$ are said to be \emph{equivalent} if $\add(T)=\add(S)$. Then by Proposition~\ref{prop:torsion-pair-from-2-term-silting} we obtain the following commutative diagram:
$$
\xymatrix{
2{\text-}\mathrm{silt}(A) \ar[rd]\ar[rr]& &\mathrm{int}{\text-}\mathrm{t}{\text-}\mathrm{str}(\mathbf{D}(A))\ar[ld]\\
 & \mathrm{tors}(\mathscr{M}_{H^0(A)})&
}
$$

Recall that the canonical projection $p\colon A\rightarrow H^{0}(A)$ is a homomorphism of dg algebras. Therefore we have the induction functor $p_*\colon\cd(A)\to\cd(H^0(A))$ and the restriction functor $p^*\colon\cd(H^0(A))\to\cd(A)$, which form an adjoint pair of triangle functors. They restrict to triangle functors $p_{\ast}\colon \per(A)\rightarrow \per(H^{0}(A))$ and  $p^{\ast}\colon \mathbf{D}(H^{0}(A))\rightarrow \mathbf{D}(A)$. Moreover, restricted on $\add_{\per(A)}(A_A)$ the induction functor $p_*$ is isomorphic to the equivalence $H^0\colon \add_{\per(A)}(A_A)\to \proj H^0(A)$ (see \eqref{eq:inducting-free-modules}); restricted on $\mathscr{M}_{H^0(A)}$ the restriction functor $p^*$ is the identity functor, thus $p^{\ast}$ is a realisation functor in the sense of \cite[Section 2.2]{ChenHanZhou19}.

\begin{proposition}
\label{prop:passing-to-the-0th-cohomology-silting-and-t-str}
\begin{itemize}
\item[(a)] {\rm(\cite[Proposition A.3]{BruestleYang13})} The induction functor $p_{\ast}\colon \per(A)\rightarrow \per(H^{0}(A))$ induces a bijection between the sets of equivalence classes of 2-term silting objects.

\item[(b)] The restriction functor $p^{\ast}\colon \mathbf{D}(H^{0}(A))\rightarrow \mathbf{D}(A)$ induces a bijection between the sets of intermediate $t$-structures, which is compatible with the bijection in Proposition~\ref{prop:bijection-between-int-t-str-and-torsion-pairs}.
\end{itemize}
\end{proposition}

\begin{proof} (b) 
Recall that $\mathscr{M}_{H^0(A)}$ embeds to both $\mathbf{D}(H^0(A))$ and $\mathbf{D}(A)$ as the heart of the standard $t$-structures and that the functor $p^\ast$ is the identity functor on $\mathscr{M}_{H^0(A)}$. Let $(\mathcal{C}^{\leq 0},\mathcal{C}^{\geq 0})$ be an intermediate $t$-structure on $\mathbf{D}(H^0(A))$. By Proposition~\ref{prop:bijection-between-int-t-str-and-torsion-pairs}, there is a torsion pair $(\mathscr{T},\mathscr{F})$ of $\mathscr{M}_{H^0(A)}$, which, by Proposition~\ref{prop:bijection-between-int-t-str-and-torsion-pairs} again, corresponds to an intermediate $t$-structure $(\cd^{\leq 0},\cd^{\geq 0})$ on $\mathbf{D}(A)$. The desired bijection is $(\mathcal{C}^{\leq 0},\mathcal{C}^{\geq 0})\mapsto (\cd^{\leq 0},\cd^{\geq 0})$. Precisely, we have
\begin{align*}
\cd^{\leq 0}&= \text{the smallest full subcategory of}~\mathbf{D}(A)~\text{which contains}~p^{\ast}(\mathcal{C}^{\leq 0})\\
&\qquad\text{and which is closed under extensions and direct summands},\\
\cd^{\geq 0}&= \text{the smallest full subcategory of}~\mathbf{D}(A)~\text{which contains}~p^{\ast}(\mathcal{C}^{\geq 0})\\
&\qquad\text{and which is closed under extensions and direct summands}.\qedhere
\end{align*}
\end{proof}

\begin{remark}
If $T$ is a silting object of $\per(A)$ which is not $2$-term, then $p_*(T)$ is in general not a silting object of $\per(H^0(A))$. For example, let $Q$ be the graded quiver
\[
\xymatrix{1\ar[r]^\alpha & 2}
\]
with $\alpha$ in degree $-1$. Denote by $S_1$ and $S_2$ the two simple modules concentrated in degree $0$, corresponding to the vertices $1$ and $2$ respectively. Then $H^0(A)=k\times k$, $T=\Sigma^3 S_1\oplus S_2$ is a silting object of $\per(A)$ but $p_*(T)=\Sigma^3 S_1\oplus \Sigma^2 S_1\oplus S_2$ is not a silting object of $\per(H^0(A))$.
\end{remark}

We summarise the above results in a commutative diagram.

\begin{theorem}
\label{thm:silting-torsionpair-tstr-along-projection}
There is a commutative diagram:

$$\xymatrix{
2{\text-}\mathrm{silt}(A) \ar[dd]\ar[rd]\ar[rr]& &\mathrm{int}{\text-}\mathrm{t}{\text-}\mathrm{str}(\mathbf{D}(A))\ar[ld]\\
 & \mathrm{tors}(\mathscr{M}_{H^0(A)})& \\
2{\text-}\mathrm{silt}(H^{0}(A)) \ar[rr]\ar[ru] & &\mathrm{int}{\text-}\mathrm{t}{\text-}\mathrm{str}(\mathbf{D}(H^{0}(A)))\ar[uu]\ar[lu]}
$$

\begin{proof} Since $(p_{\ast},p^{\ast})$ is an adjoint pair, there is an isomorphism for any $X\in\cd(H^0(A))$
\begin{center}
$\Hom_{\cd(H^0(A))}(p_{\ast}(T),\Sigma^i X)\cong\Hom_{\cd(A)}(T,\Sigma^i p^{\ast}(X))$,
\end{center}
which implies the commutativity of the left triangle. The commutativity of other triangles has been shown above.
\end{proof}
\end{theorem}

\section{The silting theorem} 
\label{s:the-silting-theorem}

In this section we will prove the silting theorem in the more general setting of non-positive dg algebras and then specialise it to algebras. Fix one of the following settings:
\begin{itemize}
\item[(1)] $\mathbf{D}=\cd^b$, $\mathscr{M}=\Mod$;
\item[(2)] $k$ is a noetherian ring, $\mathbf{D}=\cd_{fg}$, $\mathscr{M}=\mod$;
\item[(3)] $k$ is a field, $\mathbf{D}=\cd_{fd}$, $\mathscr{M}=\fdmod$.
\end{itemize}

\subsection{An abstract version}
\label{ss:silting-theorem-abstract-version}

Let $\Lambda$ and $\Gamma$ be non-positive dg $k$-algebras, $A=H^0(\Lambda)$ and $B=H^0(\Gamma)$. Assume that there is an adjoint pair of triangle equivalences
\[
\xymatrix{
\cd(\Lambda) \ar@<.7ex>[d]^{\widetilde{F}} \\ 
\cd(\Gamma).\ar@<.7ex>[u]^{\widetilde{G}}
}
\]
Then $T=\widetilde{G}(\Gamma)$ is a silting object of $\per(\Lambda)$.
By Proposition~\ref{prop:restriction-of-silting-t-structure-on-subcategories}, the above equivalences restrict to an adjoint pair of triangle equivalences
\[
\xymatrix{
\mathbf{D}(\Lambda) \ar@<.7ex>[d]\\ 
\mathbf{D}(\Gamma),\ar@<.7ex>[u]
}
\]
and the equivalence $\widetilde{G}$ takes the standard $t$-structure on $\mathbf{D}(\Gamma)$ to the $t$-structure $(\mathbf{D}_T^{\leq 0},\mathbf{D}_T^{\geq 0})$.

In the rest of this subsection we assume that $T$ is $2$-term. Consider the functors
\[
\xymatrix{
\mathscr{M}_A \ar@<.7ex>[d]^{F=H^0\circ\widetilde{F}} \\ 
\mathscr{M}_B,\ar@<.7ex>[u]^{G=H^0\circ\widetilde{G}}
}\qquad
\xymatrix{
\mathscr{M}_A \ar@<.7ex>[d]^{F'=H^1\circ\widetilde{F}} \\ 
\mathscr{M}_B.\ar@<.7ex>[u]^{G'=H^{-1}\circ\widetilde{G}}
}
\]
In Section~\ref{s:2-term-silting-objects} we defined two subcategories $\mathscr{T}_T$ and $\mathscr{F}_T$ of $\mathscr{M}_A$. We have
\begin{align*}
\mathscr{T}_T=\{M\in\mathscr{M}_A\mid F'(M)=0\}, ~~& \mathscr{F}_T=\{N\in\mathscr{M}_A\mid F(N)=0\},
\end{align*}
because 
\begin{align*}
F&=H^0\circ\widetilde{F}|_{\mathscr{M}_A}=\Hom_{\cd(\Gamma)}(\Gamma,\widetilde{F}(?))|_{\mathscr{M}_A}\cong\Hom_{\cd(\Lambda)}(T,?)|_{\mathscr{M}_A}\\
F'&=H^1\circ\widetilde{F}|_{\mathscr{M}_A}=\Hom_{\cd(\Gamma)}(\Gamma,\Sigma\widetilde{F}(?))|_{\mathscr{M}_A}\cong\Hom_{\cd(\Lambda)}(T,\Sigma?)|_{\mathscr{M}_A}.
\end{align*}
As in \cite{BreazModoi20} the adjoint pair $(F,G)$ can also be described in terms of the $A$-module $H^0(T)$, which is a silting module over $A$ by \cite[Theorem 4.9]{AngeleriMarksVitoria16} (and a support $\tau$-tilting module by \cite[Theorem 3.2]{AdachiIyamaReiten14}, if $k$ is a field and $A$ is a finite-dimensional $k$-algebra). Taking $H^0$ yields an algebra homomorphism $B=\End_{\cd(\Lambda)}(T)\to \End_A(H^0(T))$, which makes $H^0(T)$ a $B$-$A$-bimodule. Moreover, there are isomorphisms
\[
F(M)\cong\Hom_{\cd(\Lambda)}(T,M)\cong\Hom_{\cd(\Lambda)}(H^0(T),M)\cong\Hom_A(H^0(T),M),
\]
which are compatible with the $B$-actions. This establishes the first isomorphism below, and the second isomorphism is obtained by adjunction.

\begin{lemma}
We have $F\cong \Hom_A(H^0(T),?)$ and $G\cong ~?\ten_B H^0(T)$.
\end{lemma}

Put
\begin{align*}
\mathscr{X}_T=\{X\in\mathscr{M}_B\mid G(X)=0\}, ~~& \mathscr{Y}_T=\{Y\in\mathscr{M}_B\mid G'(Y)=0\}
\end{align*}

\begin{theorem} 
\label{thm:silting-theorem-1-abstract-version}
\begin{itemize}
\item[(a)] $(\mathscr{T}_T,\mathscr{F}_T)$ is a torsion pair of $\mathscr{M}_A$.
\item[(b)] $(\mathscr{X}_T,\mathscr{Y}_T)$ is a torsion pair of $\mathscr{M}_B$.
\item[(c)] There are equivalences
\[
\xymatrix{
\mathscr{T}_T \ar@<.7ex>[d]^{F|_{\mathscr{T}_T}} \\
\mathscr{Y}_T,\ar@<.7ex>[u]^{G|_{\mathscr{Y}_T}}
}\qquad
\xymatrix{
\mathscr{F}_T \ar@<.7ex>[d]^{F'|_{\mathscr{F}_T}} \\ 
\mathscr{X}_T,\ar@<.7ex>[u]^{G'|_{\mathscr{X}_T}}
}
\]
\end{itemize}
\end{theorem}

\begin{proof}
(a) This is the first statement of Proposition~\ref{prop:torsion-pair-from-2-term-silting}.

(c) For $Y\in\mathscr{M}_B$, we claim that the dg module $\widetilde{G}(Y)$ has cohomologies concentrated in degrees $0$ and $-1$. Thus if $Y\in\mathscr{Y}_T$, then $G(Y)\cong \widetilde{G}(Y)$. Similarly, for $M\in\mathscr{T}_T$ we have $F(M)\cong \widetilde{F}(M)$. It follows that $F|_{\mathscr{T}_T}\colon\mathscr{T}_T\to\mathscr{Y}_T$ and $G|_{\mathscr{Y}_T}\colon\mathscr{Y}_T\to\mathscr{T}_T$ are mutually quasi-inverse equivalences. The statement for $F'$ and $G'$ is proved similarly.

Now we prove the claim. By Lemma~\ref{lem:n-term-silting}, $\Sigma\Lambda$ is $2$-term with respect to $T$, so there is a triangle 
\[
\xymatrix{
\Lambda\ar[r] & T'\ar[r] & T''\ar[r] & \Sigma\Lambda
}
\]
with $T',~T''\in\add(T)$. Applying $\Hom_{\cd(\Lambda)}(?,\widetilde{G}(Y))$ to this triangle we obtain an exact sequence
\[
\xymatrix{
\Hom_{\cd(\Lambda)}(T',\Sigma^{i}\widetilde{G}(Y))\ar[r] & H^i(\widetilde{G}(Y))\ar[r] & \Hom_{\cd(\Lambda)}(T'',\Sigma^{i+1}\widetilde{G}(Y)).
}
\]
But $\Hom_{\cd(\Lambda)}(T',\Sigma^i\widetilde{G}(Y))\cong\Hom_{\cd(\Gamma)}(\widetilde{F}(T'),\Sigma^{i}Y)$ vanishes unless $i=0$ since $\widetilde{F}(T')\in\add(\Gamma)$, and similarly $\Hom_{\cd(\Lambda)}(T'',\Sigma^{i+1}\widetilde{G}(Y))$ vanishes unless $i=-1$. The claim follows.

(b) By Proposition~\ref{prop:torsion-pair-from-2-term-silting}, the $t$-structure $(\mathbf{D}_T^{\leq 0},\mathbf{D}_T^{\geq 0})$ is the HRS tilt of the standard $t$-structure on $\mathbf{D}(\Lambda)$ at the torsion pair $(\mathscr{T}_T,\mathscr{F}_T)$. Therefore $(\Sigma\mathscr{F}_T,\mathscr{T}_T)$ is a torsion pair of its heart $\cb$, by Construction~\ref{constr:HRS-tilt}. Moreover, the triangle equivalence $\widetilde{G}\colon\mathbf{D}(\Gamma)\to\mathbf{D}(\Lambda)$ takes the standard $t$-structure on $\mathbf{D}(\Gamma)$ to $(\mathbf{D}_T^{\leq 0},\mathbf{D}_T^{\geq 0})$, in particular, it restricts to an equivalence from $\mathscr{M}_B$ to $\cb$, and by the proof of (c), it takes the pair $(\mathscr{X}_T,\mathscr{Y}_T)$ to $(\Sigma\mathscr{F}_T,\mathscr{T}_T)$. It follows that $(\mathscr{X}_T,\mathscr{Y}_T)$ is a torsion pair of $\mathscr{M}_B$.
\end{proof}

\subsection{The first step}
\label{ss:silting-theorem-first-step}

Let $\Lambda$ be a non-positive dg $k$-algebra, $A=H^0(\Lambda)$, and let $T$ be a $2$-term silting object in $\per(\Lambda)$ which is $\ck$-projective, $\Gamma$ the truncated dg endomorphism algebra of $T$ (see Section~\ref{sss:truncated-dg-endo-algebra-of-a-silting}), and $B=H^0(\Gamma)=\End_{\cd(\Lambda)}(T)$.  Recall from Section~\ref{sss:truncated-dg-endo-algebra-of-a-silting} that there is an adjoint pair of triangle equivalences
\[
\xymatrix{
\cd(\Lambda) \ar@<.7ex>[d]^{\RHom_\Lambda(T,?)} \\ 
\cd(\Gamma).\ar@<.7ex>[u]^{?\lten_{\Gamma}T}
}
\]
Consider the functors
\[
\xymatrix{
\mathscr{M}_A \ar@<.7ex>[d]^{F=H^0\RHom_\Lambda(T,?)} \\ 
\mathscr{M}_B,\ar@<.7ex>[u]^{G=H^0(?\lten_\Gamma T)}
}\qquad
\xymatrix{
\mathscr{M}_A \ar@<.7ex>[d]^{F'=H^1\RHom_\Lambda(T,?)} \\ 
\mathscr{M}_B.\ar@<.7ex>[u]^{G'=H^{-1}(?\lten_\Gamma T)}
}
\]
and put
\begin{align*}
\mathscr{T}_T&=\{M\in\mathscr{M}_A\mid F'(M)=0\}, ~~ \mathscr{F}_T=\{N\in\mathscr{M}_A\mid F(N)=0\},\\
\mathscr{X}_T&=\{X\in\mathscr{M}_B\mid G(X)=0\}, ~~ \mathscr{Y}_T=\{Y\in\mathscr{M}_B\mid G'(Y)=0\}.
\end{align*}

Applying Theorem~\ref{thm:silting-theorem-1-abstract-version} we obtain
\begin{theorem} 
\label{thm:silting-theorem-1-torsion-pairs}
\begin{itemize}
\item[(a)] $(\mathscr{T}_T,\mathscr{F}_T)$ is a torsion pair of $\mathscr{M}_A$.
\item[(b)] $(\mathscr{X}_T,\mathscr{Y}_T)$ is a torsion pair of $\mathscr{M}_B$.
\item[(c)] There are equivalences
\[
\xymatrix{
\mathscr{T}_T \ar@<.7ex>[d]^{F|_{\mathscr{T}_T}} \\
\mathscr{Y}_T,\ar@<.7ex>[u]^{G|_{\mathscr{Y}_T}}
}\qquad
\xymatrix{
\mathscr{F}_T \ar@<.7ex>[d]^{F'|_{\mathscr{F}_T}} \\ 
\mathscr{X}_T,\ar@<.7ex>[u]^{G'|_{\mathscr{X}_T}}
}
\]
\end{itemize}
\end{theorem}

Following Section~\ref{sss:truncated-dg-endo-algebra-of-a-silting}, let $\widetilde{S}=\RHom_\Lambda(T,\Sigma \Lambda)=\cHom_\Lambda(T,\Sigma \Lambda)$, which is a $2$-term silting object of $\per(\Gamma)$. Moreover, $\widetilde{S}$ is naturally a dg $\Lambda$-$\Gamma$-bimodule, and there are isomorphisms of triangle equivalences:  
\[
\RHom_\Lambda(T,?)\cong \Sigma^{-1}(?\lten_\Lambda \widetilde{S})~\text{and}~
?\lten_{\Gamma}T\cong\Sigma\RHom_{\Gamma}(\widetilde{S},?).
\] 
Consider
\[
\xymatrix{
\mathscr{M}_B \ar@<.7ex>[d]^{\Phi=H^0\RHom_\Gamma(\widetilde{S},?)} \\ 
\mathscr{M}_A,\ar@<.7ex>[u]^{\Psi=H^0(?\lten_\Lambda \widetilde{S})}
}\qquad
\xymatrix{
\mathscr{M}_B \ar@<.7ex>[d]^{\Phi'=H^1\RHom_\Gamma(\widetilde{S},?)} \\ 
\mathscr{M}_A.\ar@<.7ex>[u]^{\Psi'=H^{-1}(?\lten_\Lambda \widetilde{S})}
}
\]
Then by Theorem~\ref{thm:silting-theorem-1-abstract-version}, there is a torsion pair $(\mathscr{T}_{\widetilde{S}},\mathscr{F}_{\widetilde{S}})$ of $\mathscr{M}_B$ and a torsion pair $(\mathscr{X}_{\widetilde{S}},\mathscr{Y}_{\widetilde{S}})$ of $\mathscr{M}_A$, where
\begin{align*}
\mathscr{T}_{\widetilde{S}}&=\{M\in \mathscr{M}_{B}\mid \Phi'(M)=0\}=\{M\in \mathscr{M}_{B}\mid \Hom_{\mathcal{D}(\Gamma)}(\widetilde{S},\Sigma M)=0\},\\
\mathscr{F}_{\widetilde{S}}&=\{N\in\mathscr{M}_{B}\mid \Phi(N)=0\}=\{N\in\mathscr{M}_{B}\mid \Hom_{\mathcal{D}(\Gamma)}(\widetilde{S},N)=0\},\\
\mathscr{X}_{\widetilde{S}}&=\{X\in\mathscr{M}_A\mid \Psi(X)=0\},\\
\mathscr{Y}_{\widetilde{S}}&=\{Y\in\mathscr{M}_A\mid \Psi'(Y)=0\}.
\end{align*}
The above isomorphisms of triangle functors imply isomorphisms 
\[\Phi\cong G',~\Phi'\cong G,~\Psi\cong F',~\text{and}~\Psi'\cong F.\]
Therefore we have

\begin{theorem}
\label{thm:silting-theorem-1-equivalences}
\begin{itemize}
\item[(a)]
$(\mathscr{T}_{\widetilde{S}},\mathscr{F}_{\widetilde{S}})=(\mathscr{X}_T,\mathscr{Y}_T)$ and $(\mathscr{X}_{\widetilde{S}},\mathscr{Y}_{\widetilde{S}})=(\mathscr{T}_T,\mathscr{F}_T)$.
\item[(b)] The functors $F'=\Hom_{\cd(\Lambda)}(T,\Sigma?)$ and $\Phi=\Hom_{\cd(\Gamma)}(\widetilde{S},?)$ restrict to inverse equivalences between $\mathscr{F}_T$ and $\mathscr{T}_{\widetilde{S}}$.
\item[(c)] The functors $F=\Hom_{\cd(\Lambda)}(T,?)$ and $\Phi'=\Hom_{\cd(\Gamma)}(\widetilde{S},\Sigma ?)$ restrict to inverse equivalences between $\mathscr{T}_T$ and $\mathscr{F}_{\widetilde{S}}$.
\end{itemize}
\end{theorem}

\subsection{The second step}
\label{ss:silting-theroem-second-step}
Keep the notation and assumptions as in the beginning of Section~\ref{ss:silting-theorem-first-step}. 
Since $\Sigma \Lambda$ is a 2-term silting object of $\per(\Lambda)$ with respect to $T$, there is a triangle in $\per(\Lambda)$:
\begin{align}
\label{eq:approximation-triangle-for-shift-of-Lambda}
\xymatrix{
\Lambda\ar[r]  & T'\ar[r]^f & T''\ar[r] & \Sigma \Lambda
}
\end{align}
with $T',~T''\in\add_{\per(\Lambda)}(T)$. By applying $\Hom_{\per(\Lambda)}(T,?)$, we obtain a homomorphism of projective $B$-modules
\[
\xymatrix@C=5pc{
\Hom_{\per(\Lambda)}(T,T^{\prime})\ar[r]^{\Hom(T,f)} & \Hom_{\per(\Lambda)}(T,T^{\prime\prime}).
}
\]
Consider the canonical projection $p\colon\Gamma\to B$, which induces two triangle functors $p_*\colon \per(\Gamma)\to\per(B)=K^b(\proj B)$ and $p^*\colon\mathbf{D}(B)\to\mathbf{D}(\Gamma)$. Recall that $\widetilde{S}=\RHom_\Lambda(T,\Sigma \Lambda)$ is a $2$-term silting object of $\per(\Gamma)$. Let $S=p_*(\widetilde{S})$ and $\bar{A}=\End_{\mathcal{D}(\Mod B)}(S)$.
The part (a) below says that the complex $S$ is the same as the complex $\mathbf{Q}$ in \cite[Theorem 1.1]{BuanZhou16}.

\begin{theorem} 
\label{thm:silting-theorem-2-the-silting-complex}
\begin{itemize}
\item[(a)] $S\cong \Cone(\Hom(T,f))$;
\item[(b)] $S$ is a 2-term silting complex in $K^{b}(\proj B)$;
\item[(c)] There is a surjective algebra homomorphism $\pi\colon A\rightarrow \bar{A}$. If $T$ is tilting, then $\pi$ is an isomorphism.
\end{itemize}
\end{theorem}
\begin{proof}
(a) Applying the triangle equivalence $\RHom_\Lambda(T,?)$ to the triangle \eqref{eq:approximation-triangle-for-shift-of-Lambda}, we obtain a triangle in $\per(\Gamma)$:
\[
\xymatrix@C=1pc{
\RHom_\Lambda(T,\Lambda)\ar[r]  & \RHom_\Lambda(T,T')\ar[rrr]^{\RHom_\Lambda(T,f)} &&& \RHom_\Lambda(T,T'')\ar[r] & \widetilde{S}.
}
\]
Applying $p_*$ we obtain a triangle in $K^b(\proj B)$:
\[
\xymatrix@C=1pc{
\Sigma^{-1}S\ar[r]  & p_*\RHom_\Lambda(T,T')\ar[rrrr]^{p_*\RHom_\Lambda(T,f)} &&&& p_*\RHom_\Lambda(T,T'')\ar[r] & S.
}
\]
Recall that restricted on $\add_{\per(\Gamma)}(\Gamma_\Gamma)$ the induction functor $p_*$ is isomorphic to the equivalence $H^0\colon \add_{\per(\Gamma)}(\Gamma_\Gamma)\to \proj B$. Thus the above map $p_*\RHom_\Lambda(T,f)$ is isomorphic to $H^0\RHom_\Lambda(T,f)=\Hom(T,f)$. This shows that $S\cong\Cone(\Hom(T,f))$.

(b) This follows from Proposition~\ref{prop:passing-to-the-0th-cohomology-silting-and-t-str}(a).

(c) There is a chain of algebra homomorphisms
\[
\xymatrix@C=1.1pc{
A=\End_{\per(\Lambda)}(\Lambda)\ar[r]^(0.52){\Sigma} &\End_{\per(\Lambda)}(\Sigma \Lambda)\ar[rrr]^{\RHom_\Lambda(T,?)} &&& \End_{\per(\Gamma)}(\widetilde{S})\ar[r]^(0.4){p_*} & \End_{K^b(\proj B)}(S)=\bar{A}.
}
\]
We define $\pi\colon A\to\bar{A}$ as the composition. The first two homomorphisms, induced by equivalences, are isomorphisms, and the third one, induced by $p_*$, is surjective with kernel the homomorphisms $\widetilde{S}\to\widetilde{S}$ which factors through a morphism $\Sigma P_1\to P_2$ with $P_1,P_2\in\add_{\per(\Gamma)}(\Gamma)$, by \cite[Proposition A.5]{BruestleYang13}. If $T$ is a tilting complex, then $\Hom_{\per(\Gamma)}(\Sigma\Gamma,\Gamma)\cong\Hom_{\per(\Lambda)}(\Sigma T,T)=0$, so the third homomorphism is also an isomorphism, so is $\pi$. 
\end{proof}


Now let $\bar{\Lambda}$ be the truncated dg endomorphism algebra of $S$. Then $H^0(\bar{\Lambda})=\bar{A}$. Applying Section~\ref{ss:silting-theorem-first-step} to the 2-term silting object $S$ of $K^b(\proj B)\simeq\per(B)$, we obtain two adjoint pairs
\[
\xymatrix{
\mathscr{M}_B \ar@<.7ex>[d]^{\bar{\Phi}=H^0\RHom_B(S,?)} \\ 
\mathscr{M}_{\bar{A}},\ar@<.7ex>[u]^{\bar{\Psi}=H^0(?\lten_{\bar{\Lambda}}S)}
}\qquad
\xymatrix{
\mathscr{M}_{B} \ar@<.7ex>[d]^{\bar{\Phi}'=H^1\RHom_B(S,?)}\\
\mathscr{M}_{\bar{A}},\ar@<.7ex>[u]^{\bar{\Psi}'=H^{-1}(?\lten_{\bar{\Lambda}}S)}
}
\]
There is a torsion pair $(\mathscr{T}_S,\mathscr{F}_S)$ of $\mathscr{M}_B$ and a torsion pair $(\mathscr{X}_S,\mathscr{Y}_S)$ of $\mathscr{M}_{\bar{A}}$ together with equivalences
\[
\xymatrix{
\mathscr{T}_S\ar@<.7ex>[d]^{\bar{\Phi}|_{\mathscr{T}_S}}\\ 
\mathscr{Y}_S,\ar@<.7ex>[u]^{\bar{\Psi}|_{\mathscr{Y}_S}}
}\qquad
\xymatrix{
\mathscr{F}_S \ar@<.7ex>[d]^{\bar{\Phi}'|_{\mathscr{F}_S}}\\ 
\mathscr{X}_S,\ar@<.7ex>[u]^{\bar{\Psi}'|_{\mathscr{X}_S}}
}
\]
where
\begin{align*}
\mathscr{T}_S&=\{M\in\mathscr{M}_B\mid \bar{\Phi}'(M)=0\}, ~~ \mathscr{F}_S=\{N\in\mathscr{M}_B\mid \bar{\Phi}(N)=0\};\\
\mathscr{X}_S&=\{X\in\mathscr{M}_{\bar{A}}\mid \bar{\Psi}(X)=0\}, ~~ \mathscr{Y}_S=\{Y\in\mathscr{M}_{\bar{A}}\mid \bar{\Psi}'(Y)=0\}.
\end{align*}

\begin{theorem}
\label{thm:silting-theorem-2-equivalences}
\begin{itemize}
\item[(a)] $(\mathscr{T}_S,\mathscr{F}_S)=(\mathscr{X}_T,\mathscr{Y}_T)$;
\item[(b)] The functors $F'=\Hom_{\cd(\Lambda)}(T,\Sigma?)$ and $\pi^*\Hom_{\cd(B)}(S,?)=\pi^*\bar{\Phi}$ restrict to inverse equivalences between $\mathscr{F}_T$ and $\mathscr{T}_S$.
\item[(c)] The functors $F=\Hom_{\cd(\Lambda)}(T,?)$ and $\pi^*\Hom_{\cd(B)}(S,\Sigma ?)=\pi^*\bar{\Phi}'$ restrict to inverse equivalences between $\mathscr{T}_T$ and $\mathscr{F}_S$.
\item[(d)] $\pi^*(\mathscr{Y}_S)=\mathscr{F}_T$ and $\pi^*(\mathscr{X}_S)=\mathscr{T}_T$.
\end{itemize}
\end{theorem}

\begin{proof} 
By Theorem~\ref{thm:silting-torsionpair-tstr-along-projection} we have $(\mathscr{T}_S,\mathscr{F}_S)=(\mathscr{T}_{\widetilde{S}},\mathscr{F}_{\widetilde{S}})$.

(a) This follows from Theorem~\ref{thm:silting-theorem-1-equivalences}(a). 

(b) Since $ \mathscr{T}_{\widetilde{S}}=\mathscr{T}_S$, by Theorem~\ref{thm:silting-theorem-1-equivalences}(b) it suffices to prove that there is an isomorphism $\Phi|_{\mathscr{T}_S}\cong(\pi^*\bar{\Phi})|_{\mathscr{T}_S}$, \ie $\Hom_{\cd(\Gamma)}(\widetilde{S},?)|_{\mathscr{T}_S}\cong(\pi^*\Hom_{\cd(B)}(S,?))|_{\mathscr{T}_S}$. For this, recall that $p_*$ is left adjoint to $p^*$, and hence there is a commutative diagram for any $a\in A$ and any $M\in\cd(B)$
\[
\xymatrix{
\Hom_{\cd(B)}(S,M)\ar@{=}[r] & \Hom_{\cd(B)}(p_*(\widetilde{S}),M)\ar[r]\ar[d]_{(p_*\RHom_\Lambda(T,\Sigma a))^*} & \Hom_{\cd(\Gamma)}(\widetilde{S},p^*(M))\ar[d]_{(\RHom_\Lambda(T,\Sigma a))^*}\\
\Hom_{\cd(B)}(S,M)\ar@{=}[r] & \Hom_{\cd(B)}(p_*(\widetilde{S}),M)\ar[r] & \Hom_{\cd(\Gamma)}(\widetilde{S},p^*(M))
}
\]
where the two horizontal arrows are isomorphisms of $k$-modules. Going through the definition of the algebra homomorphism $\pi$ in the proof of Theorem~\ref{thm:silting-theorem-2-the-silting-complex}(c), we see that this exactly shows that $\pi^*\Hom_{\cd(B)}(S,M)$ and $\Hom_{\cd(\Gamma)}(\widetilde{S},p^*(M))$ are isomorphic as $A$-modules. Since $p^*|_{\mathscr{M}_B}$ is the identity functor, it follows that $(\pi^*\Hom_{\cd(B)}(S,?))|_{\mathscr{T}_S}\cong\Hom_{\cd(\Gamma)}(\widetilde{S},?)|_{\mathscr{T}_S}$.

(c) Similar to (b).

(d) Because both $\bar{\Phi}|_{\mathscr{T}_S}\colon\mathscr{T}_S\to\mathscr{Y}_S$ and $(\pi^*\bar{\Phi})|_{\mathscr{T}_S}\colon\mathscr{T}_S\to\mathscr{F}_T$ are equivalences, we deduce that $\pi^*$ restricts to an equivalence $\mathscr{Y}_S\to\mathscr{F}_T$, in particular, $\pi^*(\mathscr{Y}_S)=\mathscr{F}_T$. The other equality is proved similarly.
\end{proof}

\subsection{Special case: $\Lambda=A$}
\label{ss:special-case-algebra}
Let $A$ be a $k$-algebra, $T$ a $2$-term silting complex in $K^b(\proj A)$, $\Gamma$ the truncated dg endomorphism algebra of $T$, and $B=H^0(\Gamma)=\End_{\cd(A)}(T)$.  Fix one of the following settings:
\begin{itemize}
\item[(1)] $\mathbf{D}=\cd^b$, $\mathscr{M}=\Mod$;
\item[(2)] $k$ is a noetherian ring, $\mathbf{D}=\cd_{fg}$, $\mathscr{M}=\mod$;
\item[(3)] $k$ is a field, $\mathbf{D}=\cd_{fd}$, $\mathscr{M}=\fdmod$.
\end{itemize}

\begin{theorem} 
\label{thm:silting-theorem-algebra}
\begin{itemize}
\item[(a)] There is a torsion pair $(\mathscr{T}_T,\mathscr{F}_T)$ of $\mathscr{M}_A$ and a torsion pair $(\mathscr{X}_T,\mathscr{Y}_T)$ of $\mathscr{M}_B$ together with natural equivalences $\Hom_{\mathcal{D}(A)}(T,?)\colon\mathscr{T}_T\stackrel{\simeq}{\longrightarrow} \mathscr{Y}_T$ and $\Hom_{\mathcal{D}(A)}(T,\Sigma ?)\colon\mathscr{F}_T\stackrel{\simeq}{\longrightarrow}\mathscr{X}_T$.
\item[(b)] Let $S\cong \Cone(\Hom(T,f))$.
\item[(c)] $S$ is a 2-term silting complex in $K^{b}(\proj B)$.
\item[(d)] There is a surjective algebra homomorphism  $\pi\colon A\rightarrow \bar{A}=\End_{K^b(\proj B)}(S)$.
\item[(e)] $\pi$ is an isomorphism if and only if $T$ is tilting.
\item[(f)] The functors $\Hom_{\cd(A)}(T,?)$ and $\pi^*\Hom_{\cd(B)}(S,\Sigma ?)$ restrict to inverse equivalences between $\mathscr{T}_T$ and $\mathscr{F}_S$.
\item[(g)] The functors $\Hom_{\cd(A)}(T,\Sigma?)$ and $\pi^*\Hom_{\cd(B)}(S,?)$ restrict to inverse equivalences between $\mathscr{F}_T$ and $\mathscr{T}_S$.
\item[(h)] $\mathscr{T}_S=\mathscr{X}_T$ and $\mathscr{F}_S=\mathscr{Y}_T$; $\pi^*(\mathscr{Y}_S)=\mathscr{F}_T$ and $\pi^*(\mathscr{X}_S)=\mathscr{T}_T$.
\end{itemize}
\end{theorem}
\begin{proof}
We apply Theorems~\ref{thm:silting-theorem-2-the-silting-complex}~and~\ref{thm:silting-theorem-2-equivalences} in this special case. It remains to show that if $\pi$ is an isomorphism then $T$ is tilting. We do not have a proof different from the one given in the end of the proof of \cite[Proposition 4.1]{BuanZhou16}. 
\end{proof}

\begin{remark} Let us show that $\pi$ is exactly the map $\Phi$ defined in \cite[Proposition 4.1]{BuanZhou16}. Then in the setting (3) the statements (a)--(g) form Theorem~\ref{thm:buan-zhou-theorem} and the second statement in (h) is exactly \cite[Theorem 4.4]{BuanZhou16}.

Let $a\in A=\End_A(A)$. According to \cite[Section 4]{BuanZhou16} there is a morphism of triangles in $\per(A)$
\[
\xymatrix{
A\ar[r] \ar[d]^a & T'\ar[r]^f \ar[d]^b & T''\ar[r] \ar[d]^c & \Sigma A\ar[d]^{\Sigma a}\\
A\ar[r]  & T'\ar[r]^f & T''\ar[r] & \Sigma A
}
\]
Thus $\Sigma A\cong \Cone(f)$ in $K^b(\proj A)$  and one checks by the construction of $b$ and $c$ that under this isomorphism $\Sigma a\in\End(\Sigma A)$ corresponds to $(b,c)$. Now  applying $p_*\RHom_A(T,?)$ we see that $\pi$ is exactly the map $\Phi$.
\end{remark}

\subsection{A partial generalisation to $n$-term silting}
\label{ss:n-term-silting}

When the third author gave a talk on Theorem~\ref{thm:silting-theorem-algebra} in USTC in early 2021, Yu Ye asked whether it can be genelised to the case when $T$ is an $n$-term silting object. In this subsection we partially generalise Section~\ref{ss:silting-theorem-first-step}. We do not know whether Section~\ref{ss:silting-theroem-second-step} can be easily generalised, because it is not known whether or not $S$ is a silting object if $T$ is not $2$-term.
Keep the notation and assumptions as in the beginning of Section~\ref{ss:silting-theorem-first-step} except that here we assume that $T$ is $n$-term ($n\geq 2$).

\medskip
For $0\leq i\leq n-1$, consider the functor
\[
\xymatrix{
\mathscr{M}_A \ar@<.7ex>[d]^{F^i=H^i\RHom_\Lambda(T,?)}  \\ 
\mathscr{M}_B\ar@<.7ex>[u]^{G^i=H^{-i}(?\lten_\Gamma T)} 
}
\]
and put
\begin{align*}
\mathscr{T}^i_T&=\{M\in\mathscr{M}_A\mid F^j(M)=0~\forall~j\neq i\}, \\
\mathscr{Y}^i_T&=\{Y\in\mathscr{M}_B\mid G^j(Y)=0~\forall~j\neq i\}.
\end{align*}
The following result was established in the literature as
\begin{itemize}
\item[-] \cite[Theorem 1.16]{Miyashita86}, if we are in Setting (1) and if further $A$ is a ring and $T$ is a tilting $A$-module;
\item[-] \cite[Chapter III, Theorem 3.2]{Happel88}, if we are in Setting (3) and if further $A$ is a finite-dimensional $k$-algebra and $T$ is a tilting $A$-module;
\item[-] \cite[Corollary 2.6]{BreazModoi17}, if we are in Setting (1) and if further $A$ is a ring.
\end{itemize}  
The proof is similar to that of Theorem~\ref{thm:silting-theorem-1-torsion-pairs}(c).

\begin{theorem}
\label{thm:silting-theorem-3}
For $0\leq i\leq n-1$, there is an equivalence
\[
\xymatrix{
\mathscr{T}^i_T \ar@<.7ex>[d]^{F^i|_{\mathscr{T}^i_T}} \\
\mathscr{Y}^{i}_T\ar@<.7ex>[u]^{G^i|_{\mathscr{Y}^{i}_T}}
}
\]
\end{theorem}

Following Section~\ref{sss:truncated-dg-endo-algebra-of-a-silting}, let $\widetilde{S}=\RHom_\Lambda(T,\Sigma^{n-1} \Lambda)=\cHom_\Lambda(T,\Sigma^{n-1} \Lambda)$, which is an $n$-term silting object of $\per(\Gamma)$. Moreover, $\widetilde{S}$ is naturally a dg $\Lambda$-$\Gamma$-bimodule, and there are isomorphisms of triangle equivalences:  
\[
\RHom_\Lambda(T,?)\cong \Sigma^{1-n}(?\lten_\Lambda \widetilde{S})~\text{and}~
?\lten_{\Gamma}T\cong\Sigma^{n-1}\RHom_{\Gamma}(\widetilde{S},?).
\] 
For $0\leq i\leq n-1$, consider the functor
\[
\xymatrix{
\mathscr{M}_B \ar@<.7ex>[d]^{\Phi^i=H^i\RHom_\Gamma(\widetilde{S},?)}  \\ 
\mathscr{M}_A,\ar@<.7ex>[u]^{\Psi^i=H^{-i}(?\lten_\Lambda \widetilde{S})} 
}
\]
and put
\begin{align*}
\mathscr{T}^i_{\tilde{S}}&=\{M\in\mathscr{M}_B\mid \Phi^j(M)=\Hom_{\cd(\Gamma)}(\widetilde{S},\Sigma^j M)=0~\forall~j\neq i\}, \\
\mathscr{Y}^i_{\tilde{S}}&=\{Y\in\mathscr{M}_B\mid \Psi^j(Y)=0~\forall~j\neq i\}.
\end{align*}
The above isomorphisms of triangle equivalences imply
\[
\Phi^i\cong G^{n-1-i},~\text{and}~\Psi^i\cong F^{n-1-i}.
\]
Therefore we have

\begin{theorem}
\label{thm:silting-theorem-3-comparison}
\begin{itemize}
\item[(a)] $\mathscr{T}^i_{\widetilde{S}}=\mathscr{Y}^{n-1-i}_T$ and $\mathscr{Y}^i_{\widetilde{S}}=\mathscr{T}^{n-1-i}_T$.
\item[(b)] The functors $F^i=\Hom_{\cd(\Lambda)}(T,\Sigma^i?)$ and $\Phi^{n-1-i}=\Hom_{\cd(\Gamma)}(\widetilde{S},\Sigma^{n-1-i} ?)$ restrict to inverse equivalences between $\mathscr{T}^i_T$ and $\mathscr{T}^{n-1-i}_{\widetilde{S}}$.
\end{itemize}
\end{theorem}

\section{When is $S$ tilting?} 
\label{s:stability}

Let $A$ be a $k$-algebra and $T$ be a $2$-term silting complex in $K^b(\proj A)$. In Section~\ref{ss:special-case-algebra}, we constructed a 2-term silting complex $S$. Yu Zhou proposed the following question.

\begin{question}
\label{ques:stability}
When is the 2-term silting complex $S$ tilting?
\end{question}

Because $\Sigma A$ is a tilting complex, it follows that $\tilde{S}$ is tilting in $\per(\Gamma)$. If $T$ is tilting, then $\Gamma$ has cohomology concentrated in degree $0$, so the projection $p\colon\Gamma\to B$ is a quasi-isomorphism, and $p_*\colon\cd(\Gamma)\to\cd(B)$ is a triangle equivalence. Therefore $S$ is tilting. This recovers the second part of \cite[Corollary 4.3]{BuanZhou16}.

\begin{corollary}
\label{cor:T-tilting=>S-tilting}
If $T$ is tilting, then $S$ is tilting.
\end{corollary}

But in general the answer to Question~\ref{ques:stability} is not known. 
In this section, we show that the answer is positive in certain special cases.

\begin{theorem}
\label{thm:stability}
Keep the notation as above. Then $S$ is tilting if one of the following conditions is satisfied:
\begin{itemize}
\item[(a)] $A$ is hereditary;
\item[(b)] $T$ is a left mutation of $A$ or a right mutation of $\Sigma A$, whenever the mutation is defined;
\item[(c)] $k$ is a field and $A$ is a finite-dimensional symmetric $k$-algebra.
\end{itemize}
\end{theorem}

To prove Theorem~\ref{thm:stability}, we need the following lemmas. Clearly, $S$ is tilting if and only if $\Hom_{K^b(\proj B)}(S,\Sigma^{-1}S)=0$, because it is a 2-term silting complex.
Recall that there is a triangle in $K^b(\proj A)$
\begin{align*}
\xymatrix{
A\ar[r]  & T'\ar[r]^f & T''\ar[r] & \Sigma A
}
\end{align*}
with $T',~T''\in\add_{K^b(\proj A)}(T)$ and that $S=\Cone(\Hom(T,f))$. The following lemma, contained in the proof of \cite[Corollary 4.3]{BuanZhou16}, gives a characterisation of the morphisms in $\Hom_{K^b(\proj B)}(S,\Sigma^{-1}S)$.

\begin{lemma}
\label{lem:morphisms-in-degree-minus1}
There is an isomorphism of $k$-modules
\[
\Hom_{K^b(\proj B)}(S,\Sigma^{-1}S)\cong\{g\in\Hom_{K^b(\proj A)}(T'',T') \mid gf=0=fg\}
\]
\end{lemma}

Then we have a sufficient condition in terms of $H^0(T')$.

\begin{lemma}
\label{lem:stability-the-hereditary-case}
If $\pd H^{0}(T')\leq 1$, then $S$ is tilting.
\end{lemma}

\begin{proof}
In view of Lemma~\ref{lem:morphisms-in-degree-minus1}, we only need to show that if $g\colon T''\to T'$ is a morphism such that $fg=0=gf$, then $g=0$. The condition $gf=0$ implies that $g$ factors through $\Sigma A$:
\[
\xymatrix{
T'\ar[r]^f & T''  \ar[r]\ar[d]^g & \Sigma A \ar@{-->}[dl]\\
A  \ar[r] & T' \ar[r]^f & T''. 
}
\]
Since $T'\in \add(T)$ is a $2$-term complex, there is a triangle
\[
\Sigma H^{-1}(T^{\prime})\rightarrow T^{\prime}\rightarrow H^{0}(T^{\prime})\rightarrow \Sigma^2 H^{-1}(T^{\prime}).
\]
It follows that $\pd H^{0}(T^{\prime})\leq 1$ if and only if $H^{-1}(T^{\prime})=0$.
Thus by our assumption we have $\Hom_{K^b(\proj A)}(\Sigma A,T^{\prime})=H^{-1}(T^{\prime})=0$.
Therefore $g=0$.
\end{proof}

\begin{proof}[Proof of Theorem~\ref{thm:stability}] 

(a) If $A$ is hereditary, then $\pd H^{0}(T^{\prime})\leq 1$. So $S$ is tilting by Lemma~\ref{lem:stability-the-hereditary-case}.

(b) If $T$ is the left mutation of $A$, then $T=T_{e}\oplus eA$, where $e$ is an idempotent of $A$ such that $\add(eA)\cap\add((1-e)A)=0$, and $T_{e}$ is given by the following triangle 
\[
eA\rightarrow P\rightarrow T_{e}\rightarrow \Sigma(eA),
\]
with the first morphism is a left $\add((1-e)A)$-approximation. Thus $T^{\prime}=P\oplus (1-e)A$, which is a projective $A$-module. Then $S$ is tilting by Lemma~\ref{lem:stability-the-hereditary-case}. The proof of the other statement is dual.

(c) If $A$ is symmetric, then $T$ is tilting, and hence $S$ is tilting by Corollary~\ref{cor:T-tilting=>S-tilting}.\end{proof}

\begin{remark}
Yu Zhou pointed out that Theorem~\ref{thm:stability}(a) can be proved using the corresponding torsion pairs. By \cite[Lemma 5.5]{BuanZhou16}, $T$ is splitting, \ie the torsion pair $(\mathscr{X}_T,\mathscr{Y}_T)$ splits, namely, $(\mathscr{T}_S,\mathscr{F}_S)$ splits, \ie $S$ is separating. Then it is tilting by \cite[Proposition 5.7]{BuanZhou16}.
\end{remark}

In the following example none of the three conditions in Theorem~\ref{thm:stability} is satisfied but $S$ is tilting.

\begin{example}
Let $A$ be the algebra given by the quiver
\[
\xymatrix{
1 & 2\ar[l]_\alpha & 3\ar[l]_{\beta}
}
\]
with relation $\alpha\beta=0$. For a vertex $i$, let $S_i$ and $P_i$ denote the corresponding simple and projective module, respectively.
Take the $2$-term silting complex $T=P_1\oplus (P_2\xrightarrow{\alpha} P_1)\oplus \Sigma P_3$ (the corresponding support $\tau$-tilting module is $M=P_1\oplus S_1$).
Then the endomorphism algebra $B=\End_{K^b(\proj A)}(T)$ is the path algebra of the quiver
\[
\xymatrix{
a\ar[r]^\gamma & b & c.
}
\] 
Then the 2-term silting complex $S$ is $\Sigma P_a\oplus (P_a\xrightarrow{\gamma} P_b)\oplus P_c$.
Direct computation shows that $\Hom_{K^b(\proj B)}(S,\Sigma^{-1}S)=0$, thus $S$ is tilting.
\end{example}

\def\cprime{$'$}
\providecommand{\bysame}{\leavevmode\hbox to3em{\hrulefill}\thinspace}
\providecommand{\MR}{\relax\ifhmode\unskip\space\fi MR }
\providecommand{\MRhref}[2]{%
  \href{http://www.ams.org/mathscinet-getitem?mr=#1}{#2}
}
\providecommand{\href}[2]{#2}

\end{document}